\documentclass[11pt]{amsart}
\usepackage{amssymb, latexsym, amsmath, amsfonts, hyperref, enumitem, fancyhdr}
\usepackage{amsmath,amscd}
\usepackage{cleveref}
\usepackage{mathrsfs}
\newtheorem{thm}{Theorem}[section]
\newtheorem{cor}[thm]{Corollary}
\newtheorem{lem}[thm]{Lemma}
\newtheorem{prop}[thm]{Proposition}
\theoremstyle{definition}
\newtheorem{defn}[thm]{Definition}
\theoremstyle{remark}

\newtheorem{rem}[thm]{Remark}
\numberwithin{equation}{section}

\hypersetup{
	colorlinks,
	citecolor=black,
	linkcolor=black,
	urlcolor=blue}
\theoremstyle{plain}

	{\end{enumerate}}

\newcommand{\norm}[1]{\left\Vert#1\right\Vert}

\hypersetup{
	colorlinks,
	citecolor=blue,
	linkcolor=blue,
	urlcolor=blue}

\begin{document}
	\title{Noncommutative ergodic theorems for action of semisimple Lie groups}

	{\let\thefootnote\relax\footnote{}}
	%
	\author{Guixiang Hong}
	\author{Samya Kumar Ray}
	
	
	\email{samyaray7777@gmail.com, samya@iisertvm.ac.in}
	\email{gxhong@hit.edu.cn}
	%

	\pagestyle{headings}
\begin{abstract}
Let $G$ be a connected simple Lie group of real rank one and finite center, and let $K$ be a maximal compact subgroup. We study the families of spherical, ball, and uniform averages $(\sigma_t)_{t>0}$, $(\beta_t)_{t>0}$, and $(\mu_t)_{t>0}$ on $G$ induced by the canonical $G$-invariant metric on $G/K$, in the setting where $G$ acts by trace-preserving $*$-automorphisms on a finite von Neumann algebra $(\mathcal M,\tau)$. For the associated noncommutative $L_p$-spaces $L_p(\mathcal M)$, we consider both local and global noncommutative maximal inequalities for these averages, and  corresponding pointwise ergodic theorems in the sense of bilateral almost uniform convergence. Our approach combines a noncommutative Calder\'on transfer principle, spectral analysis for the Gelfand pair $(G,K)$ via Harish--Chandra's spherical functions, fractional integration methods, and Littlewood--Paley $g$-function estimates. This work is complemented by our results for higher-rank semisimple Lie groups, where the presence of Property~(T) and associated spectral gaps serve as the key tools in establishing Wiener-type noncommutative pointwise ergodic theorems and noncommutative $L_p$-maximal inequalities for ball and spherical averages on $G/K$ for certain class of semisimple Lie groups.

	\end{abstract}
	
	\maketitle
	\section{Introduction}
In classical ergodic theory, one of the foundational results is {Birkhoff's pointwise ergodic theorem} \cite{Bir31}, which asserts that for a measure-preserving action of the integers $\mathbb{Z}$ on a probability space, the time averages of any $f\in L_1$ converge almost everywhere to the conditional expectation of $f$ with respect to the invariant $\sigma$-algebra. A central methodological insight---already present in early work---is that such almost everywhere convergence can be deduced from an appropriate \emph{maximal inequality}. In the $\mathbb{Z}$-action case, Birkhoff's theorem follows from a weak-type $(1,1)$ bound for the {maximal operator} associated with the ergodic averages, a fact first established by Wiener \cite{Wei39}. This maximal-inequality approach has since become a central tool in proving pointwise ergodic theorems in various contexts. A major breakthrough was achieved by Lindenstrauss \cite{Lin01}, who extended this framework to {measure-preserving actions of general amenable groups}. His proof introduced the notion of {tempered F\o lner sequences} and employed an ingenious random covering lemma to establish maximal inequalities, thereby obtaining pointwise ergodic theorems for a vast class of amenable groups, including non-discrete cases.

For {non-amenable groups}, the absence of F\o lner sequences requires new ideas. In a pioneering series of works, Nevo and Stein developed maximal inequalities and pointwise ergodic theorems for certain actions of non-amenable groups, beginning with their treatment of {free group actions} \cite{NeS94} (see also \cite{N2}). Their methods combined geometric properties of Cayley graphs with harmonic analysis on spheres in the free group, exploiting the rich symmetry structure to control the associated maximal operators. In the {semisimple Lie group} setting, Nevo and later Nevo-Stein initiated a comprehensive program to establish $L_p$ maximal inequalities for natural averaging families, such as {balls} and {spheres} in the symmetric space $G/K$ \cite{Nevo94,Nevo97,NevSt97,MaNS00}. Their approach blended {spectral analysis for Gelfand pairs} $(G,K)$ with {fractional integration techniques}, yielding sharp bounds for maximal operators in both rank-one and certain higher-rank cases. These results unified and extended classical harmonic analysis methods to the ergodic-theoretic setting, opening the way for further generalizations. Subsequent refinements extended these techniques to a broad range of groups and averaging sets, including discrete subgroups of semisimple Lie groups, groups of polynomial growth, and various geometric averaging families. For a comprehensive account of these developments and their numerous applications, we refer to the survey \cite{Ne06}.

Motivated by developments in quantum mechanics, operator algebra theory and noncommutative mathematics experienced rapid growth in the mid-20th century. Ergodic theory entered the noncommutative framework early on through the study of dynamical systems on von Neumann algebras. However, a systematic investigation of pointwise convergence in this setting began with the pioneering work of Lance, followed by important contributions from Conze, Dang--Ngoc \cite{Con78}, K\"ummerer \cite{Kum78}, Yeadon \cite{Yea77}, and others. For a long time, the absence of suitable maximal inequalities in $L_p$-spaces prevented the establishment of a full noncommutative analogue of the classical pointwise ergodic theorems. A major breakthrough was achieved by Junge and Xu \cite{JunXu07}, who established a noncommutative Dunford--Schwartz maximal ergodic theorem. Their approach to handling maximal inequalities, based on Pisier's theory of vector-valued noncommutative $L_p$-spaces \cite{Pis98}, opened the door to a wide range of further results on maximal and pointwise ergodic theorems in the noncommutative setting. We refer to \cite{Ana06}, \cite{Bek08}, \cite{HLW221}, \cite{Hon16}, \cite{HoS16}, \cite{Hu08} and the references therein for more comprehensive accounts.

One of the earliest noncommutative ergodic theorems for the action of a group beyond the integer group was proved independently by Hu \cite{Hu08} and Anantharaman \cite{Ana06}, where they studied noncommutative ergodic theorems and maximal inequalities for the action of the free group on a von Neumann algebra. In the case of actions of amenable groups, the first major breakthrough came in a remarkable paper by Hong--Liao--Wang \cite{HLW221}, who employed noncommutative probabilistic methods to establish ergodic theorems for actions of groups with polynomial growth. More recently, a noncommutative analogue of Lindenstrauss's ergodic theorem was proved by Cadilhac and Wang \cite{CaW22}. Furthermore, the first author investigated maximal ergodic inequalities associated with spherical averages in the Heisenberg group, thereby generalizing the results of \cite{Hon16}.

In this article we study noncommutative ergodic theorems for the action of semisimple Lie groups. Let us describe our result for the case of simple Lie groups first. Let $G$ be a connected simple Lie group of real rank one with finite center, and let $K$ be a maximal compact subgroup. Denote by $(\sigma_t)_{t>0}$, $(\beta_t)_{t>0}$, and $(\mu_t)_{t>0}$ the spherical, ball, and uniform averages on $G$ with respect to a left-invariant Riemannian metric on the symmetric space $G/K$. We study maximal and pointwise ergodic theorems for these averaging families when $G$ acts by trace-preserving $*$-automorphisms on a finite von Neumann algebra $(\mathcal M,\tau)$. We refer to Section~\ref{prelims} for any unexplained notations. Our first major contribution is a noncommutative maximal ergodic theorem for spherical averages for the action of $SO^0(n,1)$.

\begin{thm}[Global maximal inequalities for spherical averages]\label{thm:global}
Let $G=SO^0(n,1)$ and $n>2$. Then for all $\tfrac{n}{n-1}<p\leq \infty$ we have for all $x\in L_p(\mathcal M)$
\[
\Big\|\sup_{t\geq 1}{}^+\pi(\sigma_t)x\Big\|_p\leq C_p\|x\|_p.
\]
\end{thm}

Although our arguments extend naturally to arbitrary rank-one simple Lie groups with finite center, for the sake of clarity we restrict our exposition to the prototypical case $SO_0(n,1)$. The analytic framework rests on the spectral theory of convolution operators associated with the Gelfand pair $(SO_0(n,1), SO(n))$, combined with delicate asymptotic estimates for Harish--Chandra’s spherical functions, as developed in~\cite{Nevo94, Nevo97, NevSt97}. In the noncommutative setting, the absence of pointwise control forces us to abandon classical almost-everywhere bounds in favor of carefully tailored operator-norm inequalities. For the complex interpolation step, we adopt and refine the sophisticated method introduced in~\cite{JunXu07}, which proves particularly well-suited to the present context. Moreover, in subsection~\ref{meanergodicth} we prove mean ergodic theorems for the families $(\sigma_t)_{t>0}$, $(\beta_t)_{t>0}$, and $(\mu_t)_{t>0}$.

For higher-rank semisimple Lie groups we prove the following general theorem.

\begin{thm}\label{mainthm2}
Let $(\mathcal M,\tau, G,\pi)$ be a $W^*$-dynamical system where $\mathcal M$ is a noncommutative probability space with a faithful normal tracial state $\tau$ and $G$ is a connected semisimple Lie group with finite center and no compact factors. Assume that the averages $\nu_t$ are bi-$K$-invariant, roughly monotone, uniformly H\"older continuous with exponent $a$, and satisfy $\|\pi_0(\nu_t)\|\leq B\exp(-t\theta)$, where $\theta>0$. Then the following statements hold.
\begin{enumerate}[label=\textup{(\roman*)}]
\item[(i)] $\big\|\sup\limits_{t\geq 1}\!^+\pi(\nu_t)x\big\|_p\leq C_p\|x\|_p$ for $1<p<\infty.$
	\item[(ii)] 
	Let $1<p<\infty,$ then $\pi(\nu_t)x\to \tau(x)1$, a.u. as $t\to\infty$ for all $x\in L_p(\mathcal M).$
	\item[(iii)] For all $x\in L_p(\mathcal M),$ $1<p<\infty,$
	\[\big\|\sup\limits_{t\geq 1}\!^+\exp\Big(\frac{ua\theta t}{4}\Big)\Big(\pi(\nu_t)x-\tau(x).1\Big)\big\|_r\leq B\|x\|p\] provided $p,r$ and $0<u<1$ satisfy $\frac{1}{p}=\frac{1-u}{q}$ and $\frac{1}{r}=\frac{1-u}{q}+\frac{u}{2}$ for some $1<q<\infty.$
	\item[(iv)]The averages $\pi(\nu_t)x$ converges $\tau(x).1$ in an exponential rate, i.e. for all $t>0$ and $x\in L_p(\mathcal{M})$ we have \begin{equation}\label{meanexponen}
	\|\pi(\nu_t)x-\tau(x).1\|_r\leq C\|x\|_p\exp\Big(-\frac{ua\theta t}{4}\Big)\end{equation} where $C>0$ is a constant not depending on $x.$ Moreover, we have fro all $t_0>1$
	\begin{equation}\label{meanexponen1}
	\big\|\sup\limits_{t\geq t_0}\!^+\Big(\pi(\nu_t)x-\tau(x).1\Big)\|_r\leq B\|x\|_p\exp\Big(-\frac{ua\theta t_0}{4}\Big)
	\end{equation} for all $t>t_0.$
\end{enumerate}
\end{thm}

The above theorem establishes a noncommutative analog of the main result in \cite{MaNS00} and applies to the ball averages $(\beta_t)$ for any $G$ a connected semisimple Lie group with finite center and no compact factors which has Kazhdan's property~T acting on a noncommutative probability space. We also establish (see Theorem~\ref{integerradii}) a noncommutative maximal ergodic theorem for integer radii for spherical averages for actions of connected semisimple Lie groups with finite center and non-compact factor with spectral gap acting ergodically on a noncommutative probability space. For the local maximal ergodic theorem we prove a crucial local transfer principle (Theorem~\ref{localtransfer}). This establishes local maximal ergodic inequalities ball averages for actions of any connected non-compact semisimple Lie group with finite center.

The organization of the paper is the following. Section~\eqref{prelims} recalls preliminaries on noncommutative $L_p$-spaces, maximal norms, and basic facts. Section~\eqref{sec3frn} proves the local transfer principle. Section~\eqref{factsspher} discusses various useful concepts related to the spherical averages, and Section~\eqref{$L_2$-bound associated to spherical averages} establishes the global $L_2$ maximal inequalities for spherical averages. Section~\eqref{forpodd} and Section~\eqref{nevost} are devoted to global $L_p$ maximal inequalities for spherical averages. In Section~\eqref{higherrank} we deal with the global $L_p$ maximal inequalities and almost uniform convergence for higher-rank semisimple Lie groups.

	\section{Preliminaries}\label{prelims}

	\subsection{Noncommutaive $L_p$-spaces}
	
	Let $\mathcal M$ be a von Neumann algebra with normal, semifinite, faithful state $\tau_{\mathcal{M}}$ with separable predual. Sometimes we write only $\tau$ when there is no chance of confusion. Unless specified, we always work with von Neumann algebras of this kind. The unit in $\mathcal M$ is denoted by $1_{\mathcal M}.$ However, when the context is clear, we simply denote it by $1.$ The center of $\mathcal M$ is denoted by $\mathcal{Z}(\mathcal M).$ Let $\mathcal P(\mathcal M)$ denote the projection lattice of $\mathcal M.$ Denote $\mathcal M_{+}$ to be the set of all positive elements in $\mathcal M.$ For any $x\in\mathcal M_{+},$ we denote the support projection of $x$ by $s(x).$ 
	 For $1\leq p<\infty,$ we define the noncommutative $L_p$-space $L_p(\mathcal{M})$ to be the completion of $\mathcal{M}$ with respect to the norm $\|x\|_{L_p(\mathcal{M})}:=\tau(|x|^p)^{\frac{1}{p}},$ where $|x|:=(x^*x)^{\frac{1}{2}}.$ One sets $L_{\infty}(\mathcal{M})=\mathcal{M}$. For any $\sigma$-finite measure space we have a natural identification of $L_\infty(\Omega)\overline{\otimes}\mathcal M$ to the space $L_\infty(\Omega,\mathcal M)$ which is the space of all essentially bounded $\mathcal M$-valued weakly$^*$ measurable functions, with the trace $\int\otimes \tau.$ We have the identification for $L_p(L_\infty(\Omega)\overline{\otimes} \mathcal M)$ as the Bochner space $L_p(\Omega;L_p( \mathcal M))$ for $1\leq p<\infty.$ 

 Let us denote $L_0(\mathcal M)$ to be the set of all closed densely defined $\tau$-measurable operators on $\mathcal H$ affiliated with $\mathcal M$. A sequence $(x_n)_{n\geq 1}\subseteq L_0(\mathcal M)$ is said to \textit{converge in measure} to $x\in L_0(\mathcal M)$ if for all $\epsilon>0$, $\lim\limits_{n\to\infty}\tau(e_\epsilon^{\perp}(|x_n-x|))=0,$ where $e_\epsilon^{\perp}(y)=\chi_{(\epsilon,\infty)}(y)$ for any $y\in L_0(\mathcal M)_+$ and $\chi$ denotes the usual characteristic function. It is well-known that equipped with the topology of convergence in measure, $L_0(\mathcal M)$ becomes a complete topological $*$-algebra endowed with \textit{strong sum} and \textit{strong product}. Moreover, $\tau$ can be extended to $L_0(\mathcal M).$ 
  It is well-known that $L_p(\mathcal M)$ can be viewed as a subspace of $L_0(\mathcal M)$ for $1\leq p\leq \infty.$  We say that an element $x\in L_p(\mathcal M)$ is positive if and only if it is positive in $L_0(\mathcal M).$ The set of all positive elements in $L_p(\mathcal M)$ is denoted by $L_p(\mathcal M)_{+}.$ It is easy to see that $x\in L_p(\mathcal M)_{+}$ if and only if there is a sequence positive elements $(x_n)_{n\geq 1}\subseteq L_p(\mathcal M)$ for which $\|x_n-x\|_p\to 0$ as $n\to\infty.$ For $1\leq p\leq\infty,$ a linear map $T:L_p(\mathcal{M},\tau_{\mathcal{M}})\to L_p(\mathcal{M},\tau_{\mathcal{M}})$ is said to be positive if $T$ maps $L_p(\mathcal{M},\tau_{\mathcal{M}})_{+}$ to $L_p(\mathcal{M},\tau_{\mathcal{M}})_{+}.$

We need the following useful fact from \cite[Page 9]{Mei07}.
Let \( (\Omega, \mu)\) be a measure space. Suppose that \(f : \Omega \to \mathcal{M}\) is a weak-* integrable function and \(g : \Omega \to \mathbb{C}\) is an integrable function. Then 
\begin{equation}\label{csmei}
\Bigl|\int_{\Omega} f(x)\,g(x)\,d\mu(x)\Bigr|
   \;\le\;
   \left(\int_{\Omega} |f(x)|^{2}\,d\mu(x)\right)^{\tfrac{1}{2}}
   \left(\int_{\Omega} |g(x)|^{2}\,d\mu(x)\right)^{\tfrac{1}{2}}.
\end{equation}

	\subsection{Noncommutative vector-valued  $L_p$-spaces and point-wise convergence}
	
	It is well known that maximal norms on noncommutative $L_p$-spaces require special attention. This is mainly because for sequence of operators $(x_n)_{n\geq 1},$ there is no reasonable sense to the quantity $\sup_{n\geq 1}|x_n|.$ This difficulty can be overcome using Pisier's theory of noncommutative vector-valued $L_p$-spaces which was initiated for injective von Neumann algebras by Pisier \cite{Pis98} and extended to general von Neumann algebras by Junge \cite{Jun02}. For $1\leq p\leq\infty,$ let $L_p(\mathcal M;\ell_\infty)$ denote the space of all sequences $x=(x_n)_{n\geq 1}$ which admit the following factorization condition. There are  $a,b\in L_{2p}(\mathcal M)$ and uniformly bounded sequence $(y_n)_{n\geq 1}\subseteq\mathcal M$ such that $x_n=ay_nb$ for $n\geq 1.$ One defines \[\|(x_n)_{n\geq 1}\|_{L_p(\mathcal M;\ell_\infty)}:=\inf\Big\{\|a\|_{2p}\sup_{n\geq 1}\|y_n\|_{\infty}\|b\|_{2p}\Big\}\] where the infimum is taken over all possible factorization. Adopting the usual convention, we write $\|x\|_{L_p(\mathcal M;\ell_\infty)}=\big\|\sup\limits_{n\geq 1}\! ^+x_n\big\|_p.$ We remark that for any sequence $(x_n)_{n\geq 1}$ of self-adjoint operators in $L_p(\mathcal M)$ we have the following equivalent description of the above norm which is more intuitive. Let $x=(x_n)_{n\geq 1}\subseteq L_p(\mathcal M)$ be a sequence of self-adjoint operators. Then, $x\in L_p(\mathcal M;\ell_\infty)$ if and only if there exists an $a\in L_p(\mathcal M)_{+}$ such that $-a\leq x_n\leq a$ for all $n\geq 1.$ In this case, we have \[\big\|\sup\limits_{n\geq 1}\!^{+}x_n\big\|_p=\inf\{\|a\|_p:-a\leq x_n\leq a, a\in L_p(\mathcal M)_{+}\}.\] 
	
	More generally, if $\Lambda$ is any index set, we define $L_p(\mathcal M;\ell_\infty(\Lambda))$ to be the space of all $x=(x_\lambda)_{\lambda\in\Lambda}$ which can be factorized as $x_\lambda=ay_\lambda b$ with $a,b\in L_{2p}(\mathcal M),$ $(y_\lambda)_{\lambda\in\Lambda}\subseteq L_\infty(\mathcal M)$ and $\sup\limits_{\lambda\in\Lambda}\|y_\lambda\|_\infty<\infty.$ The norm on $L_p(\mathcal M;\ell_\infty(\Lambda))$ is defined by
	\[\|(x_\lambda)_{\lambda\in\Lambda}\|_{L_p(\mathcal M;\ell_\infty)}:=\inf\Big\{\|a\|_{2p}\sup\limits_{\lambda\in\Lambda}\|y_\lambda\|_\infty\|b\|_{2p}\Big\}\] where the infimum is taken over all possible factorizations.

	Following flockloric truncated description of the maximal norm is often useful.
	\begin{prop}\label{FINT}Let $1\leq p\leq\infty.$ An element $(x_\lambda)_{\lambda\in\Lambda}\subseteq L_p(\mathcal M)$ belongs to $L_p(\mathcal M;\ell_\infty(\Lambda))$ if and only if $\sup\limits_{\Lambda\supseteq J\ \text{is finite}}\big\|\sup\limits_{i\in J}\!^{+}x_i\big\|_p<\infty.$ Moreover, $\|(x_\lambda)_{\lambda\in\Lambda}\|_{L_p(\mathcal M;\ell_\infty)}=\sup\limits_{\Lambda\supseteq J\ \text{is finite}}\big\|\sup\limits_{i\in J}\!^{+}x_i\big\|_p.$
	\end{prop}
	We need the following lemmas.
	\begin{lem}\label{somepropmaximalnorm}Let $1\leq p\leq \infty$ and $T:L_p(\mathcal M)\to L_p(\mathcal M)$ be a positive bounded operator. Then for any $(x_n)\in L_p(\mathcal M,\ell_\infty)$ we have 
	\begin{equation}
	\big\|\sup\limits_{n\geq 1}\!^{+}Tx_n\big\|_p\leq \|T\|_{L_p(\mathcal M)\to L_p(\mathcal M)}\big\|\sup\limits_{n\geq 1}\!^{+}x_n\big\|_p.
	\end{equation}
	\end{lem}
\begin{proof}
Let $a\in L_p(\mathcal M)_{+}$ such that $x_n\leq a$ for all $n\geq 1.$ Clearly, this implies $Tx_n\leq Ta$ for all $n\geq 1.$ Hence as $T$ is bounded, $(Tx_n)\in L_p(\mathcal M,\ell_\infty).$ Moreover, we get $\big\|\sup\limits_{n\geq 1}\!^{+}Tx_n\big\|_p\leq \|Ta\|_p\leq \|T\|_{L_p(\mathcal M)\to L_p(\mathcal M)}\|a\|_p.$ The desired result is obtained by taking infimum over all possible $a.$
\end{proof}	

	Let $1\leq p<\infty.$ We define $L_p(\mathcal M;\ell_1)$ to be the space of all sequences $x=(x_n)_{n\geq 1}\subseteq L_p(\mathcal M)$ which admits a decomposition \[x_n=\sum_{k\geq 1}u_{kn}^*v_{kn}\] for all $n\geq 1,$ where  $(u_{kn})_{k,n\geq 1}$ and $(v_{kn})_{k,n\geq 1}$ are two families in $L_{2p}(\mathcal M) $ such that 
\[\sum_{k,n\geq 1}u_{kn}^*u_{kn}\in L_p(\mathcal M),\quad \sum_{k,n\geq 1}v_{kn}^*v_{kn}\in L_p(\mathcal M).\] In above all the series are required to converge in $L_p$-norm. We equip the space $L_p(\mathcal M;\ell_1)$   with the norm
\[\|x\|_{L_p(\mathcal M;\ell_1)}=\inf\Big\{\Big\|\sum_{k,n\geq 1}u_{kn}^*u_{kn}\Big\|_p^{\frac{1}{2}} \Big\|\sum_{k,n\geq 1}v_{kn}^*v_{kn}\Big\|_p^{\frac{1}{2}}\Big\},\] where infimum runs over all possible decompositions of $x$ described as above. For any positive sequence $x=(x_n)_{n\geq 1}\in L_p(\mathcal M;\ell_1)$ we have a simpler description of the norm as follows
\[\|x\|_{L_p(\mathcal M;\ell_1)}=\Big\|\sum\limits_{n\geq 1}x_n\Big\|_p.\]  
It is known that both $L_p(\mathcal M;\ell_\infty)$ and $L_p(\mathcal M;\ell_1)$ are Banach spaces. Moreover, we have the following duality fact.
\begin{prop}[\cite{Jun02}]\label{dulaityofmaximalnorms} Let $1<p<\infty.$ Let $\frac{1}{p}+\frac{1}{p^\prime}=1.$ Then we have isometrically $L_p(\mathcal M;\ell_1)^*=L_{p^\prime}(\mathcal M;\ell_\infty)$, with the duality relation given by \[\langle x,y\rangle=\sum_{n\geq 1}\tau(x_ny_n)\] for all $x\in L_p(\mathcal M;\ell_1)$ and $y\in L_{p^\prime}(\mathcal M;\ell_\infty).$
\end{prop}

Denote $L_p(\mathcal M; c_0)$ to be the closure of all finite sequences in $L_p(\mathcal M;\ell_\infty).$ One can identify the subspace $L_p(\mathcal M;c_0)$ to be the set of all sequences $(x_n)_{n\geq 1}\subseteq L_p(\mathcal M;\ell_\infty)$ which admits a factorization as $x_n=ay_nb,$ $a,b\in L_{2p}(\mathcal M)$ and $(y_n)_{n\geq 1}\subseteq L_\infty(\mathcal M)$ such that $\lim_{n\to \infty}\|y_n\|_\infty=0.$ We also define $L_p(\mathcal M;\ell_\infty^c)$ to be the space of all sequences $x=(x_n)_{n\geq 1}\subset L_p(\mathcal M)$ which admits a factorization of the form $x_n=y_na,$ for all $n\geq 1,$ where $a\in L_p(\mathcal M)$ and $(y_n)_{n\geq 1}\subseteq L_\infty(\mathcal M)$ with $\sup_{n\geq 1}\|y_n\|_\infty<\infty.$ One defines \[\|x\|_{L_p(\mathcal M;\ell_\infty^c)}:=\inf\big\{\|a\|_p\sup_{n\geq 1}\|y_n\|_{\infty}\big\},\] infimum being taken over all possible factorization. We denote $L_p(\mathcal M;c_0^c)$ to be the Banach space which is the completion of all finitely supported sequences in $L_p(\mathcal M;\ell_\infty^c).$ One can define the spaces $L_p(\mathcal M;c_0(\Lambda))$, $L_p(\mathcal M;\ell_\infty^c(\Lambda))$ and $L_p(\mathcal M;c_0^c(\Lambda))$ for any arbitrary index set $\Lambda.$ All the properties are analogous to the case when $\Lambda=\mathbb N.$
We refer to \cite{JunXu07}, \cite{Mu03} and \cite{DeJ04} for more details on the norms we described above. 
For the study of noncommutative individual ergodic theorems, we will also consider the a.u.
and b.a.u. convergence which were first introduced in \cite{La76} (also see \cite{Ja85}). For any projection $e\in\mathcal M$ denote $e^{\perp}:=1-e.$
	\begin{defn}Let $(x_n)_{n\geq 1}\subseteq L_0(\mathcal M)$ be a sequence and $x\in L_0(\mathcal M).$ We say that the sequence $(x_n)_{n\geq 1}$ converges to $x$ \textit{almost uniformly} (in short a.u.) if for any $\epsilon>0$ there exists a projection $e\in\mathcal M$ such that $\tau(e^{\perp})<\epsilon$ and $\lim\limits_{n\to\infty}\|(x_n-x)e\|_\infty=0.$
		
		We say $(x_n)_{n\geq 1}$ converges to $x$ \textit{bilaterally almost uniformly} (in short b.a.u.) if for any $\epsilon>0$ there exists a projection $e\in\mathcal M$ such that $\tau(e^{\perp})<\epsilon$ and $\lim\limits_{n\to\infty}\|e(x_n-x)e\|_\infty=0.$
	\end{defn}
	It follows from Egorov's theorem that in the case of classical probability spaces, the definition above are equivalent to the usual almost everywhere convergence. We mention the following proposition which is very useful for checking b.a.u. and a.u. convergence of sequences noncommutative $L_p$-spaces.
	\begin{prop}\cite{DeJ04}\label{ct1}
		\begin{enumerate}
			\item [(i)]Let $1\leq p<\infty$ and $(x_\lambda)_{\lambda\in \Lambda}$ be a family such that $(x_\lambda)_{\lambda\in \Lambda}\in L_p(\mathcal M,c_0(\Lambda)).$ Then, $x_\lambda\to 0,$ b.a.u. as $n\to \infty.$
			\item[(ii)] Let $2\leq p<\infty$ and $(x_\lambda)_{\lambda\in \Lambda}$  be a family such that $(x_\lambda)_{\lambda\in\Lambda}\in L_p(\mathcal M,c_0^c(\Lambda)).$ Then, $x_\lambda\to 0$ a.u.
		\end{enumerate}
	\end{prop}
	The following lemma will be very useful to us for showing exponential convergence of certain ergodic averages.
\begin{lem}\label{ptwiselem}\cite{HoSX19}
	\begin{itemize}
\item[(i)] Let $1\leq p<\infty$ and $(\Phi_t)_{t>0}$ a family of positive linear maps on $L_p(\mathcal M)$ so that $t\mapsto \Phi_t(x)$ is continuous for all $x\in L_p(\mathcal M).$ Assume that $(\Phi_t)_{t>0}$ is strong type $(p,p).$ If $(\Phi_t(x))_{t>0}$ converges a.u. as $t\to\infty$ on a dense subspace of $L_p(\mathcal M)$, then $(\Phi_t(x))_{t>0}$ converges a.u. as $t\to\infty$ for all $x\in L_p(\mathcal M).$
\item[(ii)] Let $1\leq p<\infty$ and $(\Phi_t)_{t>0}$ and $\Phi$ be linear maps on $L_p(\mathcal M)$ so that $t\mapsto \Phi_t(x)$ is continuous for all $x\in L_p(\mathcal M).$ Assume that there exists a dense subspace $\mathcal D\subseteq L_p(\mathcal M)$ such that $\lim\limits_{t_0\to\infty}\int_{t\geq t_0}\|\Phi_t(x)-\Phi(x)\|_p^p dt=0$ for all $x\in\mathcal D.$

If $2\leq p<\infty$ and $\|(\Phi_t(x))_{t>0}\|_{L_p(\mathcal M;\ell_\infty^c)}\leq C\|x\|_p,$ for all $x\in L_p(\mathcal M)$ for some positive constant $C,$ then $\Phi_t(x)\to \Phi(x)$ a.u. as $t\to\infty$ for all $x\in L_p(\mathcal M).$ If $1\leq p<\infty$ and  $\|(\Phi_t(x))_{t>0}\|_{L_p(\mathcal M;\ell_\infty)}\leq C\|x\|_p,$ for all $x\in L_p(\mathcal M)$ for some positive constant $C,$ then $\Phi_t(x)\to \Phi(x)$ b.a.u. as $t\to\infty$ for all $x\in L_p(\mathcal M).$ 
\end{itemize}
\end{lem}

\subsection{Noncommutative dynamical systems:}	\label{subsec1.3}
Let $G$ be a locally compact second countable unimodular group. Let $(\mathcal M,\tau)$ be a von Neumann algebra equipped with a normal semifinite faithful state $\tau$ with a seprabale predual. We denote $\operatorname{Aut}(\mathcal M)$ to be the set of all $\operatorname{w}^*$-continuous $*$-isomorphisms with the inverse being $\operatorname{w}^*$-continuous. A $W^*$-dynamical system is a quadruple $(\mathcal M,\tau, G, \pi)$, such that
\begin{itemize}
	\item[(i)] The map $\pi:G\to \operatorname{Aut}(\mathcal M)$ is a group homomorphism.
	\item[(ii)] The map $G:g\mapsto \pi(g)x$ is $\sigma$-$\operatorname{wot}$-continuous for all $x\in\mathcal M.$
	\item[(iii)]$\tau(\pi(g)x)=\tau(x)$ for all $g\in G$ and $x\in \mathcal M\cap L_1(\mathcal M).$
\end{itemize}
Let $1\leq p<\infty.$ For any $W^*$-dynamical system $(\mathcal M,\tau, G, \pi)$ we have that for all $x\in L_p(\mathcal M)$, the map $g\mapsto \pi(g)x$ is continuous from $G$ to $L_p(\mathcal M)$. Moreover, $g\mapsto \pi(g)$ is a  strongly continuous representation of group $G$ to $\operatorname{Aut}(L_p(\mathcal M)).$

Denote $M(G)$ to be the Banach space of all complex Radon measures on $G.$ For $\mu\in G$ define 
\[\pi(\mu)(x)\colon=\int_G\pi(g)xd\mu(g)\] for $x\in L_p(\mathcal M)$ and $1\leq p\leq\infty.$ Then $\|\pi(\mu)\|_{L_p(\mathcal M)\to L_p(\mathcal M)}\leq\|\mu\|_1.$ The map $\mu\mapsto \alpha(\mu)$ is a norm continuous $*$-representation of the involutive Banach algebra $M(G)$ as an algebra of operators on $L_2(\mathcal M).$ We refer \cite{Hon16} for more details. Denote $P(G)$ to be the subset of probability measures on $G$. Let $t\mapsto\nu_t$ be a weakly continuous map from $\mathbb R_{+}$ to $P(G)$.
\begin{defn}
A one-parameter family \( (\nu_t)_{t>0} \subset \mathcal{P}(G) \) is said to be a 
\emph{global} (respectively, \emph{local}) \emph{noncommutative maximal ergodic family in} \( L_p \) 
if for every trace-preserving action \( \pi \) of \( G \) on a noncommutative probability space \( (\mathcal M, \tau) \), 
there exists a constant \( c_p > 0 \) such that the following inequality holds:
\[
\left\| \sup\limits_{t > 1}\!^{+} \pi(\nu_t)x \right\|_p \leq c_p \|x\|_p 
\quad \text{(respectively, } 
\left\| \sup\limits_{0<t\leq 1}\!^{+} \pi(\nu_t)x \right\|_p \leq c_p \|x\|_p \text{)}
\]
for all \( x \in L_p(\mathcal M) \).
\end{defn}
Let $\mathcal M$ be a noncommutative probability space. For $1\leq p\leq\infty,$ denote $L_{p}^0(\mathcal M):=\{x\in L_p(\mathcal M):\tau(x)=0\}.$ Suppose we have a unitary representation $\pi:G\to L_2(\mathcal M).$ Then, clearly $ L_{2}^0(\mathcal M)$ is a closed subspace of $ L_2(\mathcal M)$ as $x\mapsto\tau(x)$ is a continuous linear functional. Moreover, if $\pi$ is induced by a dynamical system, then $L_{2}^0(\mathcal M)$ is invariant under $\pi$ as $\pi$ is trace-preserving.
\subsection{Ergodic theorem for action of group with polynomial growth}
Let $G$ be a locally compact group equipped with a right Haar measure $m$. Recall that for an invariant metric $d$ on $G$, we say that $(G, d)$ satisfies the doubling condition if there exists a constant $C>0$ such that $m(B(e,2r))\leq Cm(B(e,r))$ for all $r>0,$ where $B(g,r):=\{h\in G:d(g,h)<r\}$. We say that the balls are asymptotically invariant under
right translation if for every $g\in G$
\[\lim\limits_{r\to \infty}\frac{m(B(g,r)\Delta B(e,r))}{m(B(e,r))}=0.\]
\begin{thm}\cite{HLW221}\label{polyergodic}
Let $G$ be a locally compact group equipped with a right Haar measure $m$ equipped with a $G$-invariant metric $d$ such that $(G,d,m)$ satisfies doubling condition and balls are asymptotically invariant under right translation. Let $\alpha$ be a continuous action of $G$ on $\mathcal M$ by $\tau$-preserving automorphisms. Let $A_r$ be the averaging operators defined by 
\[A_rx:=\frac{1}{m(B(e,r))}\int_{B(e,r)}\alpha(g)xdm(g), \ x\in\mathcal{M}, \ r>0.\] Then Then $(A_r)_{>0}$ is of weak type $(1,1)$ and of strong type $(p,p)$ for $1<p<\infty$. Moreover, for all $1\leq p<\infty$ the sequence $(A_rx)_{r>0}$ converges almost uniformly for $x\in L_p(\mathcal M).$
\end{thm}

	\section{Simple rank $1$ Lie group acting on a von Neumann algebra:}\label{sec3frn}Let \( G \) be a connected simple Lie group with finite center and real rank one, and let \( K \) is a fixed maximal compact subgroup, and \( m_K \) denotes the normalized Haar probability measure on \( K \). Suppose \( A = \{a_t : t \in \mathbb{R}\} \) is a one-parameter subgroup consisting of hyperbolic translations such that the Cartan decomposition \( G = K\overline{A^+} K \) holds, where \(\overline{A^+} = \{a_t : t \geq 0\} \). For each \( t \in \mathbb{R} \), define the bi-$K$-invariant probability measure \( \sigma_t \) on \( G \) by
\[
\sigma_t = m_K * \delta_{a_t} * m_K,
\]
where \( m_K \) denotes the normalized Haar probability measure on \( K \), and \( \delta_{a_t} \) is the Dirac measure at \( a_t \). Let \( d \) be the canonical \( G \)-invariant Riemannian metric on the symmetric space \( G/K \) induced by the Cartan-Killing form. For \( t > 0 \), define the ball
\[
B_t = \{g \in G : d(gK, K) < t\},
\]
and let \( \beta_t \) be the probability measure on \( G \) defined by
\[
\beta_t = \frac{1}{\mathrm{vol}(B_t)} \cdot \chi_{B_t}(g) \, dg,
\]
where \( \chi_{B_t} \) is the indicator function of \( B_t \).

 Let us consider
    \[
    \mu_{t} = \frac{1}{t} \int_0^t \sigma_\tau \, d\tau, \quad 0 < s\leq t,
    \]
    be the uniform average of the family \( \{\sigma_\tau\} \). For \( t = 0 \), define \( \mu_0 = m_K \).
    
     Denote \( n(G) = \dim_{\mathbb{R}}(G/K) \). We denote $C_c^\infty(K \backslash G / K)$ to be all bi-$K$-invariant compactly supported smooth functions on $G$. Moreover due to Cartan decomposition, $C_c^\infty(K \backslash G / K)$ can be identified with all smooth even functions on $\mathbb{R}$ denoted by $C_c^{\infty,\text{even}}(\mathbb{R}).$ The identification is the following given any smooth even function $f:\mathbb{R}\to \mathbb{C},$ the function $F(k_1a_tk_2):=f(t)$ is clearly a function on $G$ which is smooth and bi-$K$-invariant.
    
    Consider a $W^*$-dynamical system $(\mathcal M,\tau, G, \pi)$ where \( (\mathcal M, \tau) \) is a noncommutative probability space. Let \( \pi: G \to \mathrm{Iso}(L_p(\mathcal M)) \) be the associated isometric representation. Then \( \pi(\sigma_t) \) and \( \pi(\beta_t) \) denote the averaging operators on \( L_p(\mathcal M) \) corresponding to the measures \( \sigma_t \) and \( \beta_t \), respectively. 
 \subsection{A noncommutative local Calderón trasfer principle:}  Let \( d(g,h) = d(h,g) \) be a symmetric, left-invariant (or right-invariant) metric on a locally compact second countable group \( G \), and define \( |g| := d(e, g) \). Then clearly,
\[
|gh| \leq |g| + |h|.
\] For any $r>0$ denote $B_r$ to be $\{g\in G:|g|\leq r\}$ and assume that $B_r$ is compact and $m(B_r)>0$ for all $r>0.$
Assume that \( \nu_t \), \( 0 \leq t\leq 1 \), is a family of probability measures on \( G \), all having their support contained in $B_r$. Suppose $(\mathcal M,\tau, G, \pi)$ is a $w^*$ dynamical system. For $f\in L_p(G;L_p(\mathcal M))$ we define
\begin{equation}
\pi^\prime(\nu_t)f(g):=\int_Gf(gh)d\nu_t(h).
\end{equation}
\begin{thm}\label{localtransfer}
Let $1\leq p<\infty.$ Suppose there exists a constant $C_p>0$ such that
\begin{equation}
\bigl\|\sup \limits_{0\leq t\leq 1}\!^{+}\ \pi^\prime(\nu_t)f\bigr\|_p\leq A_p\|f\|_p 
\end{equation}
for all $f\in L_p(G;L_p(\mathcal M)).$ Then there exists a constant $B_p>0$ such that
\begin{equation}
\bigl\|\sup \limits_{0\leq t\leq 1}\!^{+}\ \pi(\nu_t)(x)\bigr\|_p\leq B_p\|x\|_p 
\end{equation}
for all $x\in L_p(\mathcal M).$
\end{thm}
  \begin{proof}
  We fix \( x \in L_p(\mathcal M) \) and \( N \ge 1 \). Fix $t_,\dots,t_N\in[0,1].$  Since \( \pi(g) : L_p(\mathcal M) \to L_p(\mathcal M) \) is positive for all \( s \in G \), we see that the family \( (\pi(g) \otimes \mathrm{Id})_{s \in G} \) extends to a uniformly bounded family of maps on \( L_p(\mathcal M; \ell_{\infty}) \) (see, e.g., [HJX10, Proposition 7.3]). Hence, we may choose a constant \( C' > 0 \) such that
\[
\norm{\sup\limits_{1 \le n \le N}\!^+ \pi(\nu_{t_n}) x}_p = \norm{\sup\limits_{1 \le n \le N}\!^+ \pi(g^{-1}) \pi(g) \pi(\nu_{t_n}) x}_p \le C' \norm{\sup\limits_{1 \le n \le N}\!^+ \pi(g)\pi(\nu_{t_n}) x}_p, \quad g \in G.
\]
Then we have
\[
\norm{\sup\limits_{1\leq n\leq N}\!^+ \pi(\nu_{t_n}) x}_p^p \le \frac{C'^p}{m(B_r)} \int_{B_r}  \norm{\sup\limits_{1 \le n \le N}\!^+ \pi(g)\pi(\nu_{t_n}) x}_p^p \, dm(g). \tag{3.3}
\]

Define a function \( f \in L_p(G; L_p(\mathcal {M})) \) by
\[
  f(h) = \chi_{B_rB_r}(h) \pi(h) x, \quad h \in G.
\]
Then for all \( h \in F \),
\[
\pi(h) \pi(\nu_{t_n}) x = \int_{B_r} \pi(hg) x \, d\nu_n(g) = \int_{B_r} f(hg) \, d\nu_n(g) = \pi'(\nu_{t_n})f(h). \tag{3.4}
\]

Consider \( (\pi'(\nu_{t_n}) f)_{1 \le n \le N} \in L_p(L_{\infty}(G)\overline{ \otimes }\mathcal{M}; \ell_{\infty}) \). For any \( \varepsilon > 0 \), take a factorization \(\pi'(\nu_{t_n}) f = a F_n b \) such that \( a, b \in L_{2p}(L_{\infty}(G)\overline{ \otimes }\mathcal{M}; \ell_{\infty})) \), \( F_n \in L_{\infty}(G)\overline{ \otimes }\mathcal{M} \), and
\[
\norm{a}_{2p} \sup_{1 \le n \le N} \norm{F_n}_\infty \norm{b}_{2p} \le \norm{(\pi'(\nu_{t_n}) f)_{1 \le n \le N}}_{L_p(L_{\infty}(G)\overline{ \otimes }\mathcal{M}; \ell_{\infty})} + \varepsilon.
\]

Then we have
\begin{align*}
  &\int_G \left\|\sup\!^+_{1 \le n \le N} \pi'(\nu_{t_n}) f(h)\right\|_p^p \, dm(h) \\
  &\le \int_G \norm{a(h)}_{2p}^p \sup_{1 \le n \le N} \norm{F_n(h)}_\infty^p \norm{b(h)}_{2p}^p \, dm(h) \\
  &\le \norm{a}_{2p}^p \sup_{1 \le n \le N} \norm{F_n}_\infty^p \norm{b}_{2p}^p \\
  &\le \left(\norm{(\pi'(\nu_{t_n} f)_{1 \le n \le N}}_{L_p(L_{\infty}(G)\overline{ \otimes }\mathcal{M}; \ell_{\infty})} + \varepsilon\right)^p.
\end{align*}
Since \( \varepsilon > 0 \) is arbitrary,
\[
  \int_G \left\|\sup\limits_{1 \le n \le N}\!^+\pi'(\nu_{t_n}) f(h)\right\|_p^p \, dm(h) \le \norm{\sup\limits_{1 \le n \le N}\!^+ \pi'(\nu_{t_n}) f}_p^p.
\]

Combining (3.3), (3.4), and the assumption, we obtain:
\begin{align*}
  \norm{\sup\limits_{1 \le n \le N}\!^+\pi(\nu_{t_n} x)}_p^p &\le \frac{C'^p}{m(B_r)} \int_{B_r} \left\|\sup\limits_{1 \le n \le N}\!^+ \pi'(\nu_{t_n}) f(h)\right\|_p^p \, dm(h) \\
  &\le \frac{C'^p}{m(B_r)} \norm{\sup\limits_{1 \le n \le N}\!^+\pi'(\nu_{t_n}) f}_p^p \\
  &\le \frac{C^p C'^p}{m(B_r)} \norm{f}_p^p \\
  &= \frac{C^p C'^p}{m(B_r)} \int_{B_{2r}} \norm{\pi(h) x}_p^p \, dm(h) \\
  &\le \frac{C^p C'^p m(B_{2r})}{m(B_r)} \norm{x}_p^p.
\end{align*}
This completes the proof of the theorem.
\end{proof}   
\begin{rem}To apply Theorem \eqref{localtransfer}, we may equip the group \( G \) with a left-invariant Riemannian metric arising from a positive definite quadratic form on its Lie algebra \( \mathfrak{g} \). This form is assumed to coincide with the Killing form on the subspace \( \mathfrak{p} \), and to restrict to the negative of the Killing form on \( \mathfrak{k} \), where \( \mathfrak{g} = \mathfrak{k} \oplus \mathfrak{p} \) is the Cartan decomposition. Now suppose $(\nu_t)_{0\leq t\leq 1}$ be a family of bi-$K$-invariant measures with support contained in some ball $B_r$ for some $r>0.$ For a nice operator-valued function $f:G\to\mathcal{M}$, note that $\pi(\sigma_t)f=\pi(\sigma_t)\pi(m_K)f.$ Hence if we know the maximal inequality  \[\norm{\sup\limits_{0\leq t\leq 1}\!^+\pi(\sigma_{t}) f}_p\leq C_p\|f\|_p\] for any right-$K$-invariant function $f,$ we can conclude the same for  all $f\in L_p(G,L_p(\mathcal M)).$ This is because $\pi(\sigma_t)f=\pi(\sigma_t)\pi(m_K)f$ and $\pi(m_K)f$ is clearly right-$K$-invariant and moreover $\pi(m_K)$ is a contraction on $L_p(G,L_p(\mathcal{M})$. Now note that for $\widetilde{f}(g)=\pi(m_K)f(g)=\widetilde{f}(gK),$ $g\in G,$ we have $\pi(\sigma_t)\widetilde{f}(gK)=\int_K\widetilde{f}(gka_tK)dk,$ where $G$ acts naturally on $G/K$ as $(g,hK)\mapsto ghK.$ Thus for a nice function $\widetilde{f}:G/K\to\mathcal M$ we have that $\pi(\sigma_t)\widetilde{f}$ is nothing but the average of $\widetilde{f}$ over a sphere with center $gK$ and radius $t$ in the symmetric space $G/K.$ Hence it is enough to prove the maximal inequality 
\begin{equation}\label{localspherionsymm}
 \norm{\sup\limits_{0\leq t\leq 1}\!^+\pi(\sigma_{t})\widetilde{f}}_p\leq C_p\|\widetilde{f}\|_p,\ \text{for all}\ f\in L_p(G/K,L_p(\mathcal M)).
\end{equation}
Therefore, by Theorem \eqref{localtransfer} it is enough to establish inequality \eqref{localspherionsymm} to obtain the local maximal ergodic inequality for the family $(\sigma_t).$ 

For a right-$K$-invariant function $\widetilde{f}$ on $G$ define $\pi(\beta_t)\widetilde{f}(gK)=\frac{1}{\text{Vol}(B_t)}\int_ {gKB_t} \widetilde{f}(hK)dm_{G/K}(hK)$ be ball average over a ball of radius $t>0$ and center at $gK$. A similar computation for the ball averages will show that it is enough to prove the following inequality for the ball averages to obtain the corresponding local maximal ergodic theorem.
\begin{equation}\label{localballonsymm}
 \norm{\sup\limits_{0\leq t\leq 1}\!^+\pi(\beta_{t})\widetilde{f}}_p\leq C_p\|\widetilde{f}\|_p,\ \text{for all}\ f\in L_p(G/K,L_p(\mathcal M)).
\end{equation} 
\end{rem}
\begin{rem} In view of the above remark and Theorem 4.1 in \cite{HLW221} it is not hard to see that the ball averages for action of any connected noncompact semisimple Lie group with finite center satisfies noncommutative local maximal ergodic theorem.
\end{rem}
 \section{Sphere Averaging for \texorpdfstring{\( \mathrm{SO^0}(n,1) \)}{SO^0(n,1)}}\label{factsspher}
 \subsection{Real hyperbolic Lie group $\text{Iso}^0(H^n)$:}
	
	Let $n\geq 2$ and $h:\mathbb{R}^n\times\mathbb{R}^n\to\mathbb{R}$ be a bilinear form on $\mathbb{R}^{n+1},$ given by $h(x,y):=-\sum_{i=0}^{n-1}x_iy_i+x_ny_n.$ Let us denote $G_n$ to be $SO^0(n,1)$ which is nothing but the identity component of $O(n,1)=\{g\in GL_{n+1}:h(gx,gy)=h(x,y) \forall\ x,y\in \mathbb{R}^{n+1}\}.$ If $J$ denotes the diagonal matrix $\text{diag}(1,\dots,1,-1)$ then $O(n,1)=\{g\in GL_{n+1}:g^tJg=J\}.$ Moreover, the Lie algebra of both $O(n,1)$ and $SO^0(n,1)$ is $\mathfrak{g}_n=\{X\in M_{n+1}(\mathbb{R}):X^tJ+JX=0\}.$
	
	The bilinear form  on $\mathfrak{g}_n$ given by $B(X,Y):=\text{tr}(XY)$ is invariant under adjoint action of $G_n$ on $\mathfrak{g}_n$, and therefore it defines the Killing form (up to some scalar). The maps $\theta:X\mapsto JXJ$ and $\Theta:g\mapsto JgJ$ are called Cartan involutions on $\mathfrak{g}_n$ and $G_n$ respectively. Recall the Cartan decomposition of $\mathfrak{g}_n=\mathfrak{k}_n\oplus\mathfrak{p}_n$ where $\mathfrak{k}_n:=\{X\in\mathfrak{g}_n:\theta(X)=X\}$ and $\mathfrak{p}_n:=\{X\in\mathfrak{g}_n:\theta(X)=-X\}$  .Moreover $B|_{\mathfrak{{k}_n}\times \mathfrak{{k}_n}}$ is negative definite and $B|_{\mathfrak{{p}_n}\times \mathfrak{{p}_n}}$ is positive definite forms. The following concrete description of $\mathfrak{k}_n$ and $\mathfrak{p}_n$ can be calculated.
	\[\mathfrak{k_n}=\{\begin{pmatrix}
Y & 0 \\
0 & 0 
\end{pmatrix}:Y\in\mathfrak{so}_n(\mathbb{R})\}\ \text{and}\ \mathfrak{p}_n=\{\begin{pmatrix}
0 & v \\
v & 0 
\end{pmatrix}:v\in\mathbb{R}^n\}.	\] 

The Lie subgroup corresponding to $\mathfrak{k}_n$ is $K_n:=\{\begin{pmatrix}
U & 0 \\
0 & 1 
\end{pmatrix}:U\in SO_n(\mathbb{R})\}.$ The adjoint action of $K_n$ on $\mathfrak{p}_n$ is through the usual identity representation of as automorphism on $\mathbb{R}^n.$ Denote 

\[H_1=\begin{pmatrix}
0 &  & 1 \\
 & 0 & \\
 1 & & 0
\end{pmatrix}.\] Then $\mathfrak{a}:=\{tH_1:t\in\mathbb{R}\}$ is a maximal abelian subalgebra of $\mathfrak{p}_n.$ The corresponding Lie subgroup of $G_n$ is $A=\{a_t:t\in\mathbb{R}\}$ where \[a_t=\{\begin{pmatrix}
cosh t & & sinht \\
 & I &\\
 sinht & & cosht 
\end{pmatrix}:t\in\mathbb{R}.\] Denote $M=Z_K(A)$, the centralizer of $A$ in $K.$ Then \[M=\{\begin{pmatrix}
1 &  & 0 \\
 & U & \\
0 & & 1
\end{pmatrix}:U\in SO_{n-1}(\mathbb{R})\}.\] Consider the smooth manifold $G/K.$ Note that $\mathfrak{p}_n$ can be identified as the tangent space $T_{eK}(G/K)$. The Riemannian structure on $G/K$ is defined as the following. Take the inner product induced by the Killing form on $\mathfrak{p}_n$ and hence on $T_{eK}(G/K).$ Now for any other point $gK$ transfers the inner product to $T_{gK}(G/K)$. This makes $G/K$ a Riemannian manifold. This induces a $G_n$-invariant metric $d$ on $G/K.$ The Cartan decomposition $G_n$ tells us that $G_n=K_n\overline{A^{+}}K_n$, i.e. for any $g\in G_n$ there exists unique $t\geq 0$ and $k_1,k_2\in K$ such that $g=k_1a_tk_2.$
\subsection{Maximal ergodic theorem for $\mu_t$:}
\begin{thm}\label{thm2.1}
Let $1<p<\infty.$ Let $G_n=\text{Iso}^0(H^n)$ be a real hyperbolic Lie group, $n\geq 2.$ Then
\[\|\sup \limits_{t>0}\!^{+}\ \pi(\mu_t)x\|_p\leq c_p\|x\|_p\] for all $x\in L_p(\mathcal M).$
 \end{thm}
	\begin{proof}
	Note that it is enough to show the required estimate for $x\in L_p(\mathcal{M})_{+}.$ Let us fix such an $x.$ By definition we have \[\pi(\mu_t)x=\int_G\pi(g)xd\mu_t(g)=\int_K\int_K(\frac{1}{t}\int_{0}^t\pi(k_1a_sk_2)xds)dk_1dk_2.\] Note that as $Ad(k)\mathfrak{a}=\mathfrak{a}$ we  have that given any $t\geq 0$ there exists $k\in K$ such that $ka_tk^{-1}=a_{-t}.$ Hence by a change of variable and by the fact the $K$ is unimodular, it is easy to see that the above integral is equal to $\int_K\int_K(\frac{1}{t}\int_{0}^t\pi(k_1a_{-s}k_2)xds)dk_1dk_2.$ Therefore,
	\[\pi(\mu_t)x=\int_K\int_K(\frac{1}{t}\int_{-t}^0\pi(k_1a_{s}k_2)xds)dk_1dk_2.\] Hence we have that 
	\[\pi(\mu_t)x=\int_K\int_K(\frac{1}{2t}\int_{-t}^t\pi(k_1a_{s}k_2)xds)dk_1dk_2.\]
	Note that $\pi|_K$ is a continuous action on $\mathcal M.$ Hence $\pi(m_K)$ makes sense and is a bounded linear operator on $L_p(\mathcal M).$ Moreover, if we identify $A$ with $\mathbb R$, we obtain that $\pi(\mu_t)=\pi(m_K)\circ A_t\circ \pi(m_K)$ as bounded operators on $L_p(\mathcal M).$ Note that $\pi(m_K)$ is a positive bounded operator on $L_p(\mathcal M).$ Hence we have by Theorem \ref{polyergodic} and Lemma \ref{somepropmaximalnorm}
\[\|\sup \limits_{t>0}\!^{+}\ \pi(\mu_t)x\|_p=\|\sup \limits_{t>0}\!^{+}\ \pi(m_K)\circ A_t\circ \pi(m_K)x\|_p\leq C_p\|x\|_p. \]
	This completes the proof of the theorem.
\end{proof}
\subsection*{2.1. Reduction to $K$-invariant elements.}

To streamline the presentation, we restrict our attention in this section to the case of the real hyperbolic groups \( \mathrm{SO}^0(n,1) \) for \( n > 2 \).

In this section our focus will be on proving the strong maximal inequality for the global component. 
\subsection{$C^\infty$ vectors in $L_p(\mathcal M)$:} Define the following set.
\begin{equation}
\mathcal{D}(\mathcal{M}):=\Big\{\int_Gu(g)\pi({g})xdg:u\in C_c^\infty(G),\ x\in \mathcal{M}\Big\}. 
\end{equation}
\begin{lem}\label{lem2.1}Let $1\leq p<\infty.$ Then $\mathcal{D}(\mathcal{M})$ is dense in $L_p(\mathcal M).$
\end{lem}
\begin{proof}
Fix $x\in\mathcal{M}.$ Suppose that the sets $(U_n)_{n\geq 1}$ forms a neighbourhood basis of $1\in G$ with $U_n$s shrinking to $\{1\}$ as $n\to\infty.$ Let us take a smooth approximation of identity $(u_n)_{n\geq 1}$ of $G$ such that 
\begin{itemize}
\item $\text{supp}(u_n)\subseteq U_n$ for all $n\geq 1.$
\item $u_n\geq 0$ and $\int_Gu_ndg=1$ for all $n\geq 1.$
\end{itemize}
Notice $\|\int_Gu_n(g)\pi(g)(x)dg-x\|_p=\|\int_Gu_n(g)\pi({g})(x)dg-\int_Gu_n(g)xdg\|_p\leq \sup\limits_{g\in U_n}\|\pi_{g}x-x\|_p$ which goes to $0$ as $n\to\infty.$
\end{proof}
Define 
\begin{equation}
\mathcal{D}_K(\mathcal{M}):=\Big\{\int_0^\infty h(t)\pi(\sigma_t)xdt:h\in C_c^\infty(\mathbb{R}),\ x\in\mathcal{{M}}\Big\}
\end{equation}
\begin{lem} Let $1\leq p<\infty.$ Then the set $\mathcal{D}_K(\mathcal{M})$ is dense in $\pi(m_K)(L_p(\mathcal{M})).$
\end{lem}
\begin{proof}
Note that due to the Cartan decomposition \( G = K A K \), we have the natural identification
$
C_c^\infty(K \backslash G / K)= C_c^{\infty,\text{even}}(\mathbb{R}),$ we can start with $u_n$ as in Lemma \eqref{lem2.1} which is in $C_c^\infty(K \backslash G / K).$ hence we will be done if we can show that for any $y=\int_Gu(g)\pi(g)xdg$ with $u\in C_c^\infty(K \backslash G / K)$ and $x\in\mathcal{M},$ we have $\pi(m_K)y=\int_{0}^\infty h(t)\pi(\sigma_t)zdt$ where $h\in C_c^{\infty,\text{even}}(\mathbb{R})$ and $z\in\mathcal M.$

We apply the averaging operator \( \pi(m_K) \) to \( y\):
\[
\pi(m_K)y = \pi(m_K)(\int_Gu(g)\pi(g)xdg) = \int_K \int_G u(g)\, \pi(kg)x\, dg\, dm_K(k).
\]
Change variables \( g \mapsto k^{-1}g \), so that \( dg \) remains invariant. Then as $u$ is bi-$K$-invariant, we obtain that
\[
\pi(m_K)y = \int_K \int_G u(k^{-1}g)\, \pi(g)x\, dg\, dm_K(k) = \int_K \int_G u(g)\, \pi(g)x\, dg\, dm_K(k).
\]
 and as the Haar measure on $K$ is normalized, we have
\[
\pi(m_K)y = \int_G u(g)\, \pi(g)x\, dg =y.
\]
Therefore, \( y \) is \( K \)-invariant, and \( \pi(m_K)y = y. \).

\bigskip

Now, by the Cartan decomposition \( G = K\overline{A^+} K \), every element \( g \in G \) can be written as \( g = k_1 a_t k_2 \), where \( t \in \mathbb{R}_+ \), \( k_1, k_2 \in K \), and \( a_t \in\overline{A^+} \). Moreover, since \( u \in C_c^\infty(G) \) is bi-\( K \)-invariant, there exists an even function \( h \in C_c^\infty(\mathbb{R}) \) such that
\[
u(k_1 a_t k_2) = h(t) \quad \text{for all } k_1, k_2 \in K,\, t \in \mathbb{R}_+.
\]
Then we can write:
\[
y = \int_G u(g)\, \pi(g)x\, dg = \int_0^\infty h(t)\, \left( \int_K \int_K \pi(k_1 a_t k_2) x\, dm_K(k_1)\, dm_K(k_2) \right) \Delta(t)\, dt,
\]
where \( \Delta(t) \) is the Jacobian associated with the Cartan integration formula:
\[
\int_G f(g)\, dg = \int_K \int_0^\infty \int_K f(k_1 a_t k_2)\, \Delta(t)\, dk_1\, dt\, dk_2.
\]
Let us define \( z := x \), and observe that:
\[
\int_K \int_K \pi(k_1 a_t k_2) x\, dm_K(k_1)\, dm_K(k_2) = \pi(\sigma_t)x.
\]
Since \( x \in\mathcal M \), so is \( \pi(m_K)x \). Putting $\widetilde{h}=h\Delta,$ we conclude that
\[
y = \int_0^\infty \widetilde{h}(t)\, \pi(\sigma_t) z\, dt,
\]
for some \( z \in\mathcal M \), \( \widetilde{h} \in C_{c}^\infty(\mathbb{R}) \). Thus, $
y \in \mathcal{D}_K(\mathcal M).$
This completes the proof of the lemma.
\end{proof}
\begin{lem} Let $1\leq p\leq\infty.$ 
\begin{itemize}
\item[(i)]Then as an operator $\pi(m_K)^2=\pi(m_K)$ on $L_p(\mathcal{M}).$
\item[(ii)] $\pi(\sigma_t)\pi(m_K)=\pi(\sigma_t)$ on $L_p(\mathcal{M}).$
\end{itemize} 
\end{lem}
\begin{proof}
Let $x\in\mathcal M.$ Let us prove (i) first.
We note that
\[
\pi(m_K)^2x = \pi(m_K)(\pi(m_K)x = \int_K \left( \int_K \pi(hk)x\, dm_K(h) \right) dm_K(k).
\]

Now by a change variable and using the modularity of the Haar measure $m_K,$  we compute see that
\[
\pi(m_K)^2x = \int_K \pi(k)x\, dm_K(k),
\]
since \( \int_K dm_K(k) = 1 \). Hence,
\[
\pi(m_K)^2x = \pi(m_K)x,\ \text{for all}\ x\in\mathcal M.
\] The desired result now follow from the density of $\mathcal{M}$ in $L_p(\mathcal{M}).$

We prove (ii) now. Note that $\pi(\sigma_t)\pi(m_K)=\pi(m_K*\delta_{a_t}*m_K)\pi(m_K)=\pi(m_K)\pi(\delta_{a_t})\pi(m_K)^2.$ Hence the claim follows from (i).
\end{proof}
    \begin{lem}\label{lemma2.3}Let $1\leq p<\infty.$ Then the following are equivalent.
    \begin{itemize}
    \item[(i)] There exists a constant $C_p>0$ such that $\big\|\sup\limits_{t\geq 1}\!^+\pi(\sigma_t)x\big\|_p\leq C_p\|x\|_p$ for all $x\in L_p(\mathcal{M}).$
    \item[(ii)] There exists a constant $C_p>0$ such that $\big\|\sup\limits_{t\geq 1}\!^+\pi(\sigma_t)x\big\|_p\leq C_p\|x\|_p$ for all $x\in\mathcal{D}_K(\mathcal M).$
    \end{itemize}
     \end{lem}
\begin{proof}
We just need to prove that $(ii)$ implies $(i).$ Note the without loss of generality we may assume that (ii) is true for all $x\in\pi(m_K)(L_p(\mathcal M)).$ Notice that for any $x\in L_p(\mathcal M)$
\[\big\|\sup\limits_{t\geq 1}\!^+\pi(\sigma_t)x\big\|_p=\big\|\sup\limits_{t\geq 1}\!^+\pi(\sigma_t)(\pi(m_K)x+(I-\pi(m_k)x))\big\|_p=\big\|\sup\limits_{t\geq 1}\!^+\pi(\sigma_t)\pi(m_k)x\big\|_p.\] In above we have used part (ii) of Lemma \eqref{lemma2.3}. Hence by the hypothesis, we obtain 
\[\big\|\sup\limits_{t\geq 1}\!^+\pi(\sigma_t)x\big\|_p\leq C_p\|\pi(m_K)x\|_p\leq C_p\|x\|_p.\]
This completes the lemma.
\end{proof}
\subsection{Gelfand pair and spectral theorem:}\label{gelfandtrick}ountable group, and $(\pi, \mathcal{H})$ a strongly continuous unitary representation of $G$. The representation $\pi$ extends to a norm-continuous $*$-representation of the measure algebra $M(G)$.

Assume $(G,K)$ is a Gelfand pair, i.e., the double coset space $g^{-1}KgK$ for all $g\in G.$ Let $M(G,K)$ denote the subalgebra of bi-$K$-invariant measures in $M(G)$. Then $\pi(M(G,K))$ is commutative, and its norm closure
\[
A_\pi := \overline{\pi(M(G,K))}^{\|\cdot\|}
\]
is a commutative $C^*$-algebra. The spectrum $\sigma(A_\pi)$ consists of nonzero norm-continuous $*$-homomorphisms $\chi: A_\pi \to \mathbb{C}$. Each $\chi \in \sigma(A_\pi)$ gives rise to a bounded, real-valued, bi-$K$-invariant function $\varphi_\chi$ on $G$ defined by
\[
\varphi_\chi(g) := \chi\left(\pi(m_K * \delta_g * m_K)\right).
\]

For the Gelfand pair $(SO^0(n,1),SO(n))$, the following dichotomy was proved in \cite[Section 3]{Nevo94}
\begin{itemize}
  \item Either $\varphi_\chi(a_t) \neq 0$ for some $t > 0$, in which case $\chi$ is nontrivial on $L^1(G,K)$, and $\varphi_\chi$ satisfies the spherical identity:
  \[
  \int_K \varphi_\chi(xky) \, dk = \varphi_\chi(x)\varphi_\chi(y),
  \]
  and $\varphi_\chi$ is a spherical function.
  
  \item Or $\varphi_\chi(a_t) = 0$ for all $t > 0$, in which case $\chi$ vanishes on $L^1(G,K)$.
\end{itemize}
This dichotomy shows that the $\sigma(A_\pi)$ can be identified with s a subset of $\mathcal {S}\cup\{\delta_e\}$ where $\mathcal S$ is all bounded, real-valued, non-zero positive definite spherical functions on $G.$ Moreover, the spectral measure on $\{\delta_e\}$ is zero. This defines an injective map
\[
T: \sigma(A_\pi) \to \mathcal{S} \cup \{\delta_e\}, \quad T(\chi) := \varphi_\chi,.
\]
 One can then transfer the spectral measure from $\sigma(A_\pi)$ to the target space using the pushforward via $T$. For any nontrivial $\chi$ character. Moreover, $\langle\chi,\sigma_t\rangle=\phi_\chi(a_t).$ Hence by spectral theorem for any $\mu\in M(G,K)$ we have 
\begin{equation}\label{veryimspec}
\|\pi(\mu)(v)\|_{\mathcal H}^2=\int_{\sigma(A_\pi)}|\langle\chi,\pi(\mu)\rangle|^2dm_v(\chi),
\end{equation}
 where $m_v$ denotes the spectral measure determined by the vector $v\in\mathcal H.$    hence keeping the above identification in mind, we have 
 \begin{equation}\label{spectalmethodheart}
\|\pi(\sigma_t)(v)\|_{\mathcal H}^2=\int_{\sigma(A_\pi)}|\phi_z(a_t)|^2dm_v(z).
\end{equation}
\subsection{Spherical functions:}
It will prove convenient to parametrize the spherical functions as follows: put
$\rho_n=\frac{n-1}{2},$
and let
\begin{equation}\label{sphercialfnparam}
\varphi_z(a_t) = c_n\int_0^{\pi} (\cosh t - \sinh t \cos\theta)^{z-\rho_n} (\sin \theta)^{n-2} \, d\theta.
\end{equation}
The radial part of the Laplace--Beltrami operator on \( G/K \) (e.g., real hyperbolic space \( \mathbb{H}^n \)) is given by
\[
\Delta_{\text{rad}} = \frac{d^2}{dr^2} + (n-1) \coth r \cdot \frac{d}{dr}.
\]
Using this parametrization in \eqref{sphercialfnparam} the eigenvalue associated with \( \varphi_z \) is easily seen to be an eigenvector of $\Delta_{\text{rad}}$ with eigenvalue $ z^2-p_n^2$. First define, for $\operatorname{Re} z > 0$, the Harish-Chandra $c$-function
\[
\mathbf{c_n}(z) = c_n\int_0^{\pi} (1 - \cos \theta)^{z-\rho_n} (\sin \theta)^{n-2} \, d\theta=c_n\frac{\Gamma(z)}{\Gamma(z+\rho_n)}.
\]
Note that $\mathbf{c_n}(z)$ is defined and analytic in the half-plane $\{\operatorname{Re} z > -1\}$, punctured at $0$, since the $F$ function is analytic in $\mathbb{C} \setminus -\mathbb{N}.$

\begin{prop}\cite{Nevo94}\label{propofsphericalfn}The spherical functions satisfy the following conditions for $SO^0(n,1)$ for all $n\geq 2.$:
\begin{enumerate}
    \item \(\varphi_z = \varphi_{-z}\), and $\varphi_{\rho_n}=\varphi_{-\rho_n}\equiv 1$ and for all \(r\), $|\varphi_z(r)| \leq \varphi_{|\operatorname{Re} z|}(r)$
    
    \item \(\varphi_z\) is bounded and real-valued if and only if $z\in\Sigma,$ where \begin{equation}\label{actualspectrum}
  \Sigma:=\{i\lambda:\lambda\geq 0\}\cup\{s:0<s\leq\rho_n\} .
    \end{equation} 
    
        \item[(3)] 
    For $0 < s\leq \rho_n$, 
    \[
   0\leq\varphi_s(r)\leq (1 + \mathbf{c_n}(s)) (\cosh r)^{s-\rho_n}.\]
    
    \item[(4)] 
    For $|\operatorname{Re} z|\leq 1-\delta<1,$ $z\neq 0$  the following power series expansion holds for $0 < r < \infty$:
    \[
    \varphi_z(r) = e^{-\rho_n r} \bigl( \mathbf{c_n}(z) \psi_z(r) + \mathbf{c_n}(-z)\psi_{-z}(r) \bigr),
    \]
    where
    \[
    \psi_z(r) = \sum_{k=0}^\infty a_k(z) e^{(z - 2k) r}.
    \]
    Here, $a_0(\pm z) = 1$ and 
  $
    |a_k(\pm z)|\leq C k^q,$
    for all $k \geq 1$ and for some positive constants $C$ and $q$ depending on $\delta$.
    
    \item[(5)] $|\varphi_{i\lambda}|\leq \varphi_0(r)\leq B(1+r)e^{-\rho_nr}$ where $B$ is a positive constant depending only on the group $G_n=SO^0(n,1).$
\end{enumerate}
\end{prop}
\begin{rem}In view of discussion in Subsection \eqref{gelfandtrick} and Proposition \eqref{propofsphericalfn} we identify throughout the paper the spectrum of $A_{\pi}$ to be a subset of $\Sigma.$
\end{rem}
\subsection{Proof of the Mean Ergodic Theorems for $(\sigma_t),$ $(\mu_t)$ and $(\beta_t)$ for $SO^0(n,1)$ for $n\geq 2$:}\label{meanergodicth}
We begin by identifying the limit in the \( L_2 \)-norm of the family \( \pi(\sigma_t)(x)\), where \( \{\sigma_t\}_{t > 0} \subset M(G, K) \) is the one-parameter family of bi-\(K\)-invariant probability measures corresponding to spherical averages. We aim to establish that this family is mean ergodic. Recall by equation \eqref{veryimspec}

\begin{equation}
\|\pi(\sigma_t)(x)\|_{L_2(\mathcal{M})}^2=\int_{\sigma(A_\pi)}|\phi_z(a_t)|^2dm_x(z),
\end{equation}
where \( \mu_x \) is the spectral measure associated to the vector \( x \).

According to spectral estimates in Proposition \eqref{propofsphericalfn}, we have:
\[
|\varphi_z(t)| \leq 1 \quad \text{for all } t > 0,
\]
and moreover,
\[
\varphi_z(\pi(\sigma_t)) \to 0 \quad \text{as } t \to \infty \quad \text{for all nontrivial } \varphi_z.
\]
Therefore, by the Lebesgue Dominated Convergence Theorem, we obtain:
\[
\| \pi(\sigma_t) x\|^2 \to 0 \quad \text{as } t \to \infty,
\]
provided that \( \mu_x(\{\chi_0\}) = 0 \), where \( \chi_0 \) is the trivial character on \( \mathcal{A}_\pi \). Now note that the space of vectors invariant under each \( \pi(\sigma_t) \) coincides with the space of \( G \)-invariant vectors, since for a unit vector \( x \), the identity \( \pi(\sigma_t)x = x \) implies that \( \pi(g)x = x \) for almost all points \( g \) in the support of \( \sigma_t \). We explain this a little more. If a unit vector \( x \in L_2(\mathcal{M}) \) satisfies
\[
\pi(\sigma_t)x = \int_G \pi(g)x \, d\sigma_t(g) = x,
\]
then taking inner product with \( x \) on both sides gives
\[
\left\langle \int_G \pi(g)x \, d\sigma_t(g), x \right\rangle = \langle x, x \rangle = 1.
\]
Using linearity and continuity of the inner product, we may write:
\[
\int_G \langle \pi(g)x, x \rangle \, d\sigma_t(g) = 1.
\]
Since \( \pi(g) \) is unitary and \( x \) is a unit vector, we have \( |\langle \pi(g)x, x \rangle| \leq 1 \) for all \( g \in G \). The fact that the average of these values equals 1 implies that
\[
|\langle \pi(g)x, x \rangle| = 1 \quad \text{for } \sigma_t\text{-almost every } g \in G.
\]
Equality in the Cauchy--Schwarz inequality implies that for such \( g \), there exists \( \lambda(g) \in \mathbb{C} \) with \( |\lambda(g)| = 1 \) such that
\[
\pi(g)x = \lambda(g) x.
\]
But since
\[
\langle \pi(g)x, x \rangle = \lambda(g) \langle x, x \rangle = \lambda(g),
\]
and \( \langle \pi(g)x, x \rangle = 1 \), we conclude \( \lambda(g) = 1 \), and hence:
\[
\pi(g)x = x \quad \text{for } \sigma_t\text{-almost every } g \in \operatorname{supp}(\sigma_t).
\]
Now note that 
\[
\bigcup_t \operatorname{supp}(\sigma_t) = G.
\]
Therefore, we must have that \( x \) is a \( G \)-invariant vector. Therefore, denoting the projection onto the space of \( G \)-invariant vectors by \( E \), for any \( x \), the spectral measure of \( x - Ex \) assigns zero measure to the trivial character. It follows that
\[
\| \pi(\sigma_t)(x) - Ex \| \to 0 \quad \text{as } t \to \infty.
\]

To prove the mean ergodic theorem for \( \mu_t \) and \( \beta_t \), we use the same argument as before, and it remains only to check that \( |\varphi_{z}(\mu_t)| \) and \( |\varphi_{z}(\beta_t)| \) converge to zero as \( t \to \infty \), for all nontrivial characters \( \varphi_z \). This has already been proved in \cite[Section 4]{Nevo94}.

\section{$L_2$-bound associated to spherical averages for $SO^0(n,1)$ for $n\geq 3$:}\label{$L_2$-bound associated to spherical averages}
Elements of the set $\mathcal{D}(\mathcal M)$ are \( C^\infty \) in the following sense:  
Given a right-invariant vector field \( D \in \operatorname{Lie}(G) \), the map
\[
w_D(t) := \pi(\exp(tD))w
\]
is a \( C^\infty \) map from \( \mathbb{R} \) to \( L_p(\mathcal M) \) for all $1\leq p<\infty.$  
More explicitly, the quantity
\[
\frac{1}{h} \left( \pi(\exp(t + h)D)w - \pi(\exp(tD))w \right)
\]
converges in the $L_p$-norm as \( h\to 0 \), and the limit is, by definition, the first derivative \( w_D'(t) \) of \( w_D \).  
The function \( w_D' \) itself can be differentiated similarly, as can all successive derivatives.

Assume again that \( G \) is a real hyperbolic Lie group acting on $\mathcal M$, and let \( x \in\mathcal{D}(\mathcal M) \). In that case, the map
\[
t \mapsto \pi(\sigma_t)x
\]
is of the form
\[
t \mapsto \pi(m_K) \, \pi(\exp(tH)) \, \pi(m_K)x,
\]
where \( H \in \mathfrak{a}\) i.e in the Lie algebra of \( A \). Hence, it is a \( C^\infty \) map from \( \mathbb{R}_+ \) to \( L_p(\mathcal M)\), for all $1\leq p<\infty$ namely,
\[
\frac{1}{h} \left( \pi(\sigma_{t+h})x - \pi(\sigma_t)x \right)
\]
converges in \( \|.\|_p \)-norm to a function we denote by $
\frac{d}{dt} \pi(\sigma_t)f.$ For $p=\infty$ the above mentioned derivatives will make sense as $w^*$ topology. Rest of this section we work with $SO^0(n,1)$ where $n\geq 3.$

 Let $x\in \mathcal{D}(\mathcal M)_{+}.$ Using integration by parts we see that  for all $N\geq 1$
\[\int_1^N t\frac{d}{dt}\pi(\sigma_t)x=N\pi(\sigma_N)x-\pi(\sigma_1)x-\int_1^N\pi(\sigma_t)xdt.\] Therefore, we can rewrite the above equality as
\[\pi(\sigma_N)x-\pi(\mu_N)x-\frac{1}{N}\Big(\pi(\sigma_1)x-\pi(\mu_1)x\Big)=\frac{1}{N}\int_1^Nt\Big(\frac{d}{dt}\pi(\sigma_t)x\Big)dt.\] Since the right hand side of the above is self-adjoint, we have by \eqref{csmei}
\begin{align}
\pi(\sigma_N)x - \pi(\mu_N)x - \frac{1}{N}\Big(\pi(\sigma_1)x - \pi(\mu_1)x\Big)
&\leq \left( \left| \frac{1}{N} \int_1^N t \left( \frac{d}{dt} \pi(\sigma_t)x \right) dt \right|^2 \right)^{\frac{1}{2}} \notag \\
&\leq \frac{1}{N} \left( \int_1^N t \, dt \right)^{\frac{1}{2}} \left( \int_1^N t \left| \frac{d}{dt} \pi(\sigma_t)x \right|^2 dt \right)^{\frac{1}{2}}.
\end{align}

Let us define $g_1(x):=\Big(\int_1^\infty t\Big|\frac{d}{dt}\pi(\sigma_t)x\Big|^2dt\Big)^{\frac{1}{2}}.$ Then we have that 
\begin{equation}
\bigl\|\sup_{N \ge 1}\!^+ \Big(\pi(\sigma_N)x-\pi(\mu_N)x-\frac{1}{N}\Big(\pi(\sigma_1)x-\pi(\mu_1)x\Big)\Big)\bigr\|_2\leq \|g_1(x)\|_2.
\end{equation}
The above equation yields via triangle inequality
\begin{equation}
\bigl\|\sup_{N \ge 1}\!^+\pi(\sigma_N)x\bigr\|_2\leq\bigl\|\sup_{N \ge 1}\!^+\pi(\mu_N)x\bigr\|_2+\|g_1(x)\|_2+2\|x\|_2.
\end{equation}
Therefore, due to Theorem \eqref{thm2.1}, for the desired maximal inequality for the global part corresponding to the spherical averages, we need to show only
\begin{equation}\label{desdl2in}
\|g_1(x)\|_2\leq C\|x\|_2.
\end{equation}
Note that by linearity of trace and Fubini's theorem we obtain \[
\|g_1(x)\|_2^2=\tau(\int_1^\infty t\Big|\frac{d}{dt}\pi(\sigma_t)x\Big|^2dt)=\int_1^\infty t\|\frac{d}{dt}\pi(\sigma_t)x\|_2^2dt. 
\] We need the following lemma of which the proof follows verbatim as in \cite[Lemma 4]{Nevo94}.
\begin{lem}
For any $x\in\mathcal{D}(\mathcal M)$ we have 
\[\Big\|\frac{d}{dt}\pi(\sigma_t)x\Big\|_2^2=\int_{\Sigma}\Big|\frac{d}{dt}\phi_z(t)\Big|^2dm_x(z)\] where $m_x$ is the spectral measure associated with $x$ and $\Sigma$ as in \eqref{actualspectrum}.
\end{lem}
In view of the above lemma the desired $L_2$-inequality will be proved provided we show that $\int_1^\infty t\Big|\frac{d}{dt}\varphi_z(t)\Big|^2dt$ is a bounded function over $z\in \Sigma$ as $\int_\Sigma dm_x(z)=\|x\|_2^2.$ This necessary information has been proved in \cite[Section 6]{Nevo94}. hence the desired maximal inequality is proved for $\mathcal{D}(\mathcal M)$. Therefore, we have the following theorem by using Lemma \eqref{lem2.1}.
\begin{thm}\label{l2bdsphere}
There exists a constant $C>0$ such that for all $x\in L_2(\mathcal{M})$ we have \[\bigl\|\sup_{t \ge 1}\!^+\pi(\sigma_t)x\bigr\|_2\leq C\|x\|_2.\]
\end{thm}
\section{Noncommuttaive ergodic theorem for action of $SO^0(2m+1)$}\label{forpodd}
Let $\zeta(t)$ be a nonnegative $C^{\infty}$ function on $[0, \infty)$ which is identically $0$ on $[0, 1/4]$ and identically $1$ on $[1/2, \infty)$.  Let $x\in\mathcal{D}_K(\mathcal M)$ and define the operator-valued $C^\infty$ function $F_x(t):=\zeta(t)\pi(\sigma_t)x$ for $t\geq 0.$ To this end, define the following fractional integrals.

For $\alpha \in \mathbb{C}$ with $\operatorname{Re}(\alpha) > 0$, and an operator-valued $C^\infty$-function $F_x(t):=\zeta(t)\pi(\sigma_t)x,$ where $x\in\mathcal{D}(\mathcal{M}),$ consider the Riemann--Liouville fractional integral:
\[ I^\alpha(F_x)(t) = \frac{1}{\Gamma(\alpha)} \int_0^t \zeta(s)(t - s)^{\alpha - 1} F(s) \, ds. \]

Finally, consider the normalized fractional integrals:
\begin{equation}\label{Nalpha}
 N^\alpha(F_x)(t) = t^{-\alpha} I^\alpha(F)(t) = \frac{1}{\Gamma(\alpha)} t^{-\alpha} \int_0^t (t - s)^{\alpha - 1} F_x(s) \, ds=\frac{1}{\Gamma(\alpha)}\int_0^1(1-s)^{\alpha-1}\zeta(st)F_x(st)ds. 
 \end{equation} As in \cite{Nevo97} the functions $\alpha\mapsto N^\alpha(F_x)(t)$ can be analytically continued in the following way. Applying integration by parts
\begin{align*}
N^\alpha(F_x)(t) &= \frac{1}{\Gamma(\alpha)} \int_0^{1/2}(1-s)^{\alpha-1}F({st})\ ds \\
&+\frac{1}{\Gamma(\alpha+1)}(\frac{1}{2})^{\alpha}F_x(\frac{t}{2})\ + \frac{1}{\Gamma(\alpha + 1)} \int_{1/2}^1 (1 - s)^{\alpha} \frac{d}{ds}(F_x(st)) \, ds.
\end{align*}

Thus right-hand side defines an analytic function in the region $\operatorname{Re}(\alpha) > -1$, which we take as the definition of $N^\alpha(F)(t)$ in that region (recall that $1/\Gamma(\alpha)$ is entire). Iterating the procedure, $N^\alpha(F_x)(t)$ (and hence also $I^\alpha(F_x)(t)$) is analytically continued to an entire function.  One can check as in the classical case \cite{Nevo97} that 
\begin{itemize}
\item[1.] $I^\alpha(I^\beta)=I^{\alpha+\beta}.$
\item[2.] $I^0F_x=F_x.$
\item[3.]$I^1(F_x)(t)=\int_0^tF_x(s)ds$ and $I^{-k}(F_x)(t)=F_x^{(k)}(t)$ for all $k\geq 0.$
\item[4.]$N^0(F_x)(t)=\pi(\sigma_t)x$ for $t\geq \frac{1}{2}.$
\item[5.]$N^1F_x(t)=\pi(\mu_t)x-\frac{1}{t}B_\zeta x,$ where $B_\zeta(x):=\int_0^{\frac{1}{2}}(1-\zeta(s))\pi(\sigma_s)x\ ds,$ for $t\geq \frac{1}{2}$
\item[6.]$N^{-k}(F_x)(t)=t^k\frac{d^k}{ds^k}|_{s=t}(\pi(\sigma_s)x)$ for all integers $k\geq 0$ and $t\geq\frac{1}{2}.$
\end{itemize}
\subsection{The Littlewood-Paley $g$-functions:} For $x\in \mathcal{D}(\mathcal M)$ and $k\in\mathbb{N}$ we define
\begin{equation}
g_k(x)=\Big(\int_1^\infty t^{2k-1}\Big|\frac{d^k}{ds^k}|_{s=t}(\pi(\sigma_s)x)\Big|^2dt\Big)^{\frac{1}{2}}
\end{equation}
\begin{prop}\label{prop4.1}Suppose for all $x\in \mathcal{D}(\mathcal{M})$ we have   $\|g_j(x)\|_2\leq C_j\|x\|_2$ for all $1\leq j\leq k.$ Then we have that 
\begin{equation}
\bigl\|\sup_{t \geq 1}\!^+ N^{-k+1}(F_x)(t)\bigr\|_2\leq C_k^\prime\|x\|_2+\|\frac{d^{k-1}}{ds^{k-1}}|_{s=1}\pi(\sigma_s)x\|_2
\end{equation}
where $C_k^\prime$ depends only on $k$ and $C_1,\dots,C_k.$
\end{prop}
\begin{proof}
Let $x\in\mathcal{D}(\mathcal M)_{+}$. We proceed by induction on $k.$ For $k=1$ we refer to Section \eqref{$L_2$-bound associated to spherical averages}. Hence consider $k\geq 2.$ Using integration by parts we obtain
\begin{equation}\label{ineqgk}
\int_1^ts^k\frac{d^k}{ds^k}\pi(\sigma_s)xds=t^k\frac{d^{k-1}}{ds^{k-1}}{|_{s=t}}-\frac{d^{k-1}}{ds^{k-1}}{|_{s=1}}\pi(\sigma_s)x-\int_1^tks^{k-1}\frac{d^{k-1}}{ds^{k-1}}\pi(\sigma_s)xds
\end{equation}
Dividing by $t$ and rearranging we get
\begin{align}\label{triangleg2}
\bigl\|\sup_{t \ge 1}\!^+\, t^{k-1} \tfrac{d^{k-1}}{dt^{k-1}} \pi(\sigma_t)x \bigr\|_2
&\leq \left\| \left. \tfrac{d^{k-1}}{ds^{k-1}} \right|_{s=1} \pi(\sigma_s)x \right\|_2 \notag + \biggl\| \sup_{t \ge 1}\!^+ \, \frac{1}{t} \int_1^t s^k \tfrac{d^k}{ds^k} \pi(\sigma_s)x \, ds \biggr\|_2 \notag \\
&\quad + \biggl\| \sup_{t \ge 1}\!^+ \, \frac{k}{t} \int_1^t s^{k-1} \tfrac{d^{k-1}}{ds^{k-1}} \pi(\sigma_s)x \, ds \biggr\|_2.
\end{align}
Note that by \eqref{csmei} we have 
\begin{equation}\label{ineqgk-1}
\frac{1}{t} \int_1^t s^k \tfrac{d^k}{ds^k} \pi(\sigma_s)x \, ds\leq \Big|\frac{1}{t} \int_1^t s^k \tfrac{d^k}{ds^k} \pi(\sigma_s)x \, d\Big|\leq \frac{1}{t}\Big(\int_1^ts\ ds\Big)^{\frac{1}{2}}\Big(\int_1^ts^{2k-1}\Big|\frac{d^k}{ds^k}\pi(\sigma_s)x\Big|^2\Big)^{\frac{1}{2}}\leq g_k(x).
\end{equation}
Similarly, we also obtain
\begin{equation}
\frac{k}{t} \int_1^t s^{k-1} \tfrac{d^{k-1}}{ds^{k-1}} \pi(\sigma_s)x \, ds\leq C_k^{\prime\prime} g_{k-1}(x).
\end{equation}
Hence clearly, by \eqref{ineqgk-1}, \eqref{ineqgk} and \eqref{triangleg2} we obtain that 
\[\bigl\|\sup_{t \ge 1}\!^+\, t^{k-1} \tfrac{d^{k-1}}{dt^{k-1}} \pi(\sigma_t)x \bigr\|_2\leq C_k^\prime\|x\|_2+\|\frac{d^{k-1}}{ds^{k-1}}|_{s=1}\pi(\sigma_s)x\|_2\]
where $C_k^\prime$ depends only on $k$ and $C_1,\dots,C_k.$ This completes the proof of then proposition.
\end{proof}
\begin{prop}\label{prophighorder}
For $x\in \mathcal{D}$ we have 
\begin{equation}
\|g_k(x)\|_2^2=\int_{\Sigma}\int_1^\infty t^{2k-1}|\varphi_z^{(k)}(t)|^2dtd\nu_x(z),
\end{equation}
where $\nu_x$ is the spectral measure determined by $x$ on the spectrum $\Sigma.$
\end{prop}
\begin{proof}
Note that we have by using Fubini's theorem
\begin{align}
\|g_k(x)\|_2^2
  &= \tau\!\left(
        \int_{1}^{\infty} t^{2k-1}
        \Bigl|
          \tfrac{d^{\,k}}{dt^{\,k}}\pi(\sigma_t)\,x
        \Bigr|^{2}
        \,dt
     \right) \\ \nonumber
  &= \int_{1}^{\infty} t^{2k-1}
     \Bigl\|
       \tfrac{d^{\,k}}{dt^{\,k}}\pi(\sigma_t)\,x
     \Bigr\|_2^{2}\,dt \\ \label{spectralhigher}
  &= \int_{1}^{\infty} t^{2k-1}
     \int_{\Sigma}
       \Bigl|
         \tfrac{d^{\,k}}{dt^{\,k}}\varphi_{z}(t)
       \Bigr|^{2}
       \,d\nu_x(z)\,dt \\ 
  &= \int_{\Sigma}
       \int_{1}^{\infty} t^{2k-1}
         \bigl|\varphi_{z}^{(k)}(t)\bigr|^{2}\,dt
       \,d\nu_x(z).\nonumber
\end{align}
In above equation we have also used that $\|\frac{d^k}{dt^k}\pi(\sigma_t)x\|_2^2=\int_{\Sigma}\Big|\frac{d^k}{dt^k}\varphi_z(t)\Big|^2d\nu_x(z)$, the proof of which follows along the same line as in \cite{Nevo94, Nevo97}.
\end{proof}
We follow the notation $N^\alpha(F_x)(t)\equiv N_t^\alpha(x)$ where $x\in \mathcal{D}(\mathcal M).$

\begin{prop}\label{oddcriticall2esti}
Let \( G = \mathrm{SO}^0(n,1) \), where \( n \geq  3 \) is odd. Then for \( \alpha = a + ib \in \mathbb{C} \) and \( x \in L_2(\mathcal M) \), the maximal inequality
\[
\bigl\|\sup_{t \ge 1}\!^+N_t^\alpha x \bigr\|_{L_2(\mathcal M)} \leq C_a\, e^{\pi |b|} \|x\|_{L_2(\mathcal M)}
\]
holds, provided \( a > -\frac{n}{2} + 1 \).
\end{prop}
The above proposition will require the following three results, all of which hold for arbitrary \( n\geq  3 \), whether even or odd. The first establishes a maximal inequality for a differentiation operator involving a whole number of derivatives. Let \( G = \mathrm{SO}^0(n,1) \), and assume \( n\geq 2m + 1 \geq 3 \). Then the following proposition holds:

\begin{prop}\label{4.4}
There exists a constant $C_m>0$ such that for all $x\in\mathcal{D}(\mathcal{M})$
\[
\bigl\|\sup_{t \ge 1}\!^+N_t^{-m+1} x \bigr\|_{L_2(\mathcal M)} \leq C_m \|x\|_{L_2(\mathcal M)}.
\]
\end{prop}

\begin{prop}\label{4.6}
For every \( a > -m + \frac{1}{2} \), there exists constant $C_a>0$ such that
\[
\bigl\|\sup_{t \ge 1}\!^+N_t^{a} x\|_{L_2(\mathcal M)} \leq C_a \|x\|_{L_2(\mathcal M)},
\quad \text{for all } x \in \mathcal{D}(\mathcal M).
\]
\end{prop}

\begin{prop}\label{4.7}
The maximal inequality
\[
\bigl\|\sup_{t \ge 1}\!^+N_t^{-m+1} x\|_{L_2(\mathcal M)} \leq C_m \|x\|_{L_2(\mathcal M)},
\quad \text{for all } x \in \mathcal{D}(\mathcal M)
\]
implies that for \( \alpha = a + ib \in \mathbb{C} \) with \( a > -m + \frac{1}{2} \), the following vertical line maximal inequality also holds:
\[
\bigl\|\sup_{t \ge 1}\!^+N_t^{a+ib}x\bigr\|_{L_2(\mathcal M)} \leq C_a\, e^{\pi |b|} \|x\|_{L^2(M)},
\quad \text{for all } x\in L_2(\mathcal M).
\]
\end{prop}
The following proposition can be found in \cite{Nevo97}.
\begin{prop}
For \( G = \mathrm{SO}^0(n,1) \), with \( n= 2m + 1 \), the expressions
\[
| \varphi_z^{(k)}(1) | \quad \text{and} \quad \int_1^\infty t^{2k - 1} |\varphi_z^{(k)}(t)| \, dt
\]
are bounded on the real spectrum of the spherical Fourier transform on \( A_\pi \), i.e. $\Sigma$ for all \( 1\leq k\leq m \).
\end{prop}
\begin{proof}[Proof of Proposition \eqref{4.4}]
The proof of Proposition \eqref{4.4} follows easily from Proposition \eqref{prophighorder} and Proposition \eqref{prop4.1}.
\end{proof}
\begin{proof}[Proof of Proposition \eqref{4.6}]
Writing $n=2m+1,$ fix $0<\delta<\frac{1}{2}$ and $x\in\mathcal{D}(\mathcal{M})_{+}.$ By semigroup property we see that \[
\begin{aligned}
I^{-\delta - m + 1} F_x(t) 
&= I^{-\delta} \big( I^{-m + 1} F_x \big)(t) \\
&= \frac{1}{\Gamma(-\delta)} \int_0^{t/2} (t - s)^{-\delta - 1} I^{-m + 1} F_x(s) \, ds \\
&\quad + \frac{1}{\Gamma(1 - \delta)} \left( \frac{t}{2} \right)^{-\delta} I^{-m + 1} F_x\left( \frac{t}{2} \right) \\
&\quad + \frac{1}{\Gamma(1 - \delta)} \int_{t/2}^{t} (t - s)^{-\delta} \frac{d}{ds} I^{-m + 1} F_x(s) \, ds
\end{aligned}
\]
Multiplying by $t^{\delta+m-1}$ we obtain 
\[
\begin{aligned}
N^{-\delta - m + 1} F_x(t) 
&= \frac{1}{\Gamma(-\delta)} \int_0^{1/2} (1 - s)^{-\delta - 1} N^{-m + 1} F_x(st) \, ds \\
&\quad + \frac{1}{\Gamma(1 - \delta)} \left( \frac{1}{2} \right)^{-\delta} N^{-m + 1} F_x\left( \frac{t}{2} \right) \\
&\quad + \frac{1}{\Gamma(1 - \delta)} \int_{t/2}^{t} (t - s)^{-\delta} t^{\delta - \frac{1}{2}} t^{m - \frac{1}{2}} I^{-m} F_x(s) \, ds=I_t(x)+II_t(x)+III_t(x).
\end{aligned}
\]
We now estimate all the three terms individually. Let us start with the third term. Note that since all the terms are self-adjoint, we obtain for $1\leq t\leq 2s$ by \eqref{csmei}
\begin{equation}
\begin{aligned}
\frac{1}{\Gamma(1 - \delta)} \int_{t/2}^{t} (t - s)^{-\delta} t^{\delta - \frac{1}{2}} t^{m - \frac{1}{2}} I^{-m} F_x(s) \, ds 
&\leq \left| \frac{1}{\Gamma(1 - \delta)} \int_{t/2}^{t} (t - s)^{-\delta} t^{\delta - \frac{1}{2}} t^{m - \frac{1}{2}} I^{-m} F_x(s) \, ds \right| \\
&\leq C_\delta \left( \int_{1/2}^1 (1 - s)^{-2\delta} \, ds \right)^{\frac{1}{2}} 
\left( \int_{t/2}^t t^{2m - 1} |I^{-m} F_x(s)|^2 \, ds \right)^{\frac{1}{2}}\\
&\leq C_\delta 2^{2m-1}\Big(\int_{\frac{1}{2}}^{\infty}s^{2m-1}|F_x^{(m)}(s)|^2\ ds\Big)^{\frac{1}{2}}=C_\delta \widetilde{g}_m(x).
\end{aligned}
\end{equation}
Note that $\|\widetilde{g}_m(x)^2-{g}_m(x)^2\|_1=\|\int_{1/2}^1s^{2m-1}|\frac{d^m}{ds^m}\eta(s)\pi(\sigma_s)x|^2\ ds\|_1=\int_{\Sigma}\int_{1/2}^1s^{2m-1}|\frac{d^m}{ds^m}\varphi_z(s)|^2\ ds\ d\nu_{x}(z).$ Moreover, the set $\{\int_{1/2}^1s^{2m-1}|\frac{d^m}{ds^m}\varphi_z(s)|^2\ ds:z\in \Sigma\}$ is bounded \cite{Nevo97}. Therefore, we obtain as in Proposition \eqref{4.4} that $\|\widetilde{g}_m(x)\|_2\leq C_m\|x\|_2.$ Hence \[\bigl\|\sup_{t \ge 1}\!^+III_t(x)\bigr\|_2\leq C_m\|x\|_2.\] Note we have \[\bigl\|\sup_{t \ge 1}\!^+II_t(x)\bigr\|_2\leq C_\delta\bigl\|\sup_{t \ge 1/2}\!^+N^{-m+1}F_x(t)\bigr\|_2.\]Now proceeding as in Proposition \eqref{prop4.1} we obtain again that \[\bigl\|\sup_{t \ge 1/2}\!^+N^{-m+1}F_x(t)\bigr\|_2\leq C_m\bigl(\|\widetilde{g}_m(x)\|_2+\|F_x^{(m-1)}(1/2)\|_2\bigr).\] Note that by the fact $\{|\frac{d^m}{ds^m}|_{s=\frac{1}{2}}\varphi_z(s)|:z\in\Sigma\}<\infty,$ we obtain from the above equations that \[\bigl\|\sup_{t \ge 1}\!^+II_t(x)\bigr\|_2\leq C_m\|x\|_2.\] Now for the first part, we have \[\bigl\|\sup\limits_{t\geq 1}\!^+I_t(x)\!\|_2\leq C_\delta\bigl\|\sup\limits_{t\geq 0}\!^+N_t^{-m+1}(x)\!\|_2\] notice that by following Proposition \eqref{prop4.1} it is enough to show that \[\|\Big(\int_{1/4}^1 t^{2j-1}\bigl|\frac{d^j}{ds^j}|_{s=t}\eta(s)\pi(\sigma_s)x\bigr|^2\Big)^{\frac{1}{2}}\|_2\leq C_j\|x\|_2\] for all $1\leq j\leq m.$ However, it follows from \cite[Section 2]{Nevo97} that $\sup\{\int_{1/4}^1|t^{2m-1}\frac{d^m}{ds^m}\eta(t)\varphi_z(t)|^2dt:z\in\Sigma\}$ is finite. Hence we also obtain  as in Proposition \eqref{4.4}\[\bigl\|\sup\limits_{t>0}\!^+I_t(x)\!\|_2\leq C_m\|x\|_2.\] This completes the proof of the theorem.
\end{proof}
\begin{proof}[Proof of Proposition \eqref{4.7}]
Suppose $\alpha=a+ib$ with $ a>0.$ We prove that \[\bigl\|\sup_{t \ge 1}\!^+N_t^{a+ib}x\bigr\|_{L_2(\mathcal M)} \leq C_a\, e^{\pi |b|}\bigl\|\sup_{t \ge 1}\!^+N_t^{a}x\bigr\|_{L_2(\mathcal M)},
\quad \text{for all } x \in L_2(\mathcal M)_{+}.
\] To see this consider $y\in L_{p^\prime}(\mathcal{M})_{+}.$
Using the standard estimates for the Gamma function, we have:
\[
\left| \frac{\Gamma(a)}{\Gamma(a+ib)} \right| \leq C_a\exp(\pi|b|).
\] Now we observe that \[|\tau(N_t^{a+ib}(x)y)|=|\frac{1}{\Gamma(\alpha)} \int_{0}^1 (1-s)^{a+ib-1} \tau(F_x(st)y)  \, ds|\leq  C_a\exp(\pi|b|)\tau(N_t^axy).\]
For  finitely many $t_1,\dots,t_n>0$ we notice that for any positive sequence $(y_{t_i})\in L_{p^\prime}(\mathcal M)_{+}$ with $\|(y_{t_i})\|_{L_p(\mathcal M'\ell_1^n)}\leq 1$ we have
\begin{equation}|\sum\tau(N_{t_k}^{a+ib}(x)y_{t_k})|\leq C_a\exp(\pi|b|)\sum\tau(N_{t_k}^{a}xy_k). 
\end{equation}
 Moreover, any element in the unit ball of  $L_p(\mathcal{M},\ell_1^n)$ can be written as sum of eight positive elements in the same unit ball. Hence the claim is proved by duality Proposition \eqref{dulaityofmaximalnorms} and Proposition \eqref{FINT}.
 
 Now when $-m+\frac{1}{2}<\text{Re}\ a\leq 0,$ we can repeat the argument above with the proof of Proposition \eqref{4.6} by replacing $-\delta$ by $-\delta+ib$ to obtain \[\bigl\|\sup_{t \ge 1}\!^+N_t^{a+ib}x\bigr\|_{L_2(\mathcal M)}\leq \bigl|\frac{\Gamma(a)}{\Gamma(a+ib)}\bigr|\bigl\|\sup_{t \ge 1}\!^+N_t^{-m}x\bigr\|_{L_2(\mathcal M)}.\]
 Therefore, from our assumption and foregoing estimate of $\Gamma$ function we obtain the desired estimate whenever $a\in-\mathbb{N}.$ For $a\in-\mathbb{N}$ and $a>-m+\frac{1}{2}$ we can do integration by parts to obtain the desired estimate as above (see \cite{Stein70}. Hence the proof is completed by Proposition \eqref{4.6}.
\end{proof}

We need the following simple lemma.
\begin{lem}\label{oddcriticalLpestimate}
Let $G=SO^0(n,1),$ $n\geq 2$. Then for all $1<p<\infty$ we have that 
\[\bigl\|\sup_{t \ge 1}\!^+N_t^{1+ib}x\bigr\|_p\leq C_p(G)\|x\|_p\]
\end{lem}
\begin{proof}
Again using duality as in the first part of Proposition \eqref{4.7} can see that 
\[\bigl\|\sup_{t \ge 1}\!^+N_t^{1+ib}x\bigr\|_p\leq C\bigl\|\sup_{t \ge 1}\!^+N_t^{1}x\bigr\|_p.\]
Note that $N_t^1x=\pi(\mu_t)-\frac{1}{t}B_{\zeta}x.$ Therefore, by Theorem \eqref{thm2.1} the lemma is proved.
\end{proof}
\begin{thm}\label{maximalsphericalforodd}Let $G=SO^0(n,1),$ $n\geq 3$ is odd. Then for all $p>\frac{n}{n-1}$ we have the maximal inequality
\begin{equation}
\bigl\|\sup_{t \ge 1}\!^+\pi(\sigma_t)x\bigr\|_p\leq C_p(G)\|x\|_p
\end{equation}
\end{thm}
\begin{proof}[Proof of Theorem \eqref{maximalsphericalforodd}]
It suffices to prove this result in the case \( p \leq 2 \), since the results in the case \( p > 2 \) can be obtained by complex interpolation from the result in the case \( p = 2 \) and the trivial result in the case \( p = \infty \). For \( \theta \in [0, 1] \) and \( q \in [1, 2] \), define the index \( p \) as a function \( p(\theta, q) \) on \( [0, 1] \times [1, 2] \) satisfying the following relationship:
\[
\frac{1}{p(\theta, q)} = \frac{1 - \theta}{2} + \frac{\theta}{q}.
\]
It is easy to check that for fixed \( \theta \in (0, 1) \), the function \( p(\theta, q) \) is strictly increasing in \( q \), and \( p(\theta, (1,2]) = (p(\theta, 1), 2] \). 

We define \( \theta_0 \) such that
\[
0 = (1 - \theta_0)(-n + \tfrac{1}{2}) + \theta_0,
\]
and set \( p_0 = p(\theta_0, 1) \). It is easy to check that
\[
p_0 = \frac{2n - 1}{2n +1}.
\]

Let \( p \) be such that \( p_0 < p \leq 2 \). Then we can find \( \theta < \theta_0 \), \( q \in (1, 2] \), \( b \in \mathbb{R} \) with \( -n + \tfrac{1}{2} < b \), and \( c \in \mathbb{R} \) with \( c \geq 1 \) such that
\[
\frac{1}{p} = \frac{1 - \theta}{2} + \frac{\theta}{q}, \quad 0 = (1 - \theta)b + \theta c.
\]

Let \( x \in L_p(\mathcal{M}) \) and \( y = (y_r)_{r>0} \) with \( \|x\|_p < 1 \) and \( \|(y_r)\|_{L_{p}(\ell_1)} < 1 \). Define
\[
f(z) = u |x|^{\frac{p(1-z)}{2} + \frac{pz}{q}}, \quad \forall z \in \mathbb{C},
\]
where \( x = u |x| \) is the polar decomposition of \( x \).

By complex interpolation of the space \( L^{p'}(\ell_\infty) \) (see, e.g., Proposition 2.5 of \cite{JunXu07}), there exists a function \( g = (g_r)_{r>0} \) continuous on the strip \( \{z : 0 \leq \operatorname{Re}(z) \leq 1\} \) and analytic in its interior such that \( g(\theta) = y \) and
\[
\sup_{t \in \mathbb{R}} \max\left\{ \|g(it)\|_{L^2(\ell_1)}, \|g(1 + it)\|_{L^{q'}(\ell_1)} \right\} < 1.
\]

Now define
\[
G(z) = \exp\left( \delta\left(z^2 - \theta^2 \right) \right) \sum_{r} \tau\left( N^{(1 - z)b + zc} {F}_{f(z)}(r) g_r(z) \right),
\]
where \( \delta > 0 \) is a constant to be chosen.

The function \( G \) is analytic in the open strip. Applying Proposition \eqref{4.7}, we estimate
\begin{align*}
|G(it)| &\leq \exp\left( \delta(-t^2 - \theta^2) \right) \left\| \left( N^{b + i(tc - tb)} {F}_f(it)(r) \right)_r \right\|_{L^2(\ell_\infty)} \cdot \|g(it)\|_{L^2(\ell_1)} \\
&\leq C_a \exp\left( \delta(-t^2 - \theta^2) \right) \exp\left( \pi |tc - tb| \right) \|f(it)\|_2\\
&\leq C_a\exp\left( \delta(-t^2 - \theta^2) \right) \exp\left( \pi |tc - tb| \right).
\end{align*}

Similarly, by Lemma \eqref{oddcriticalLpestimate}, we have
\begin{align*}
|G(1 + it)| &\leq \exp\left( \delta(-(1 + it)^2 - \theta^2) \right) \left\| \left( N^{c + i(tc - tb )}{F}_{f(1 + it)}(r) \right)_r \right\|_{L^q(\ell_\infty)} \cdot \|g(1 + it)\|_{L^{q'}(\ell_1)} \\
&\leq C_{a,q} \exp\left( \delta(1 - t^2 - \theta^2) \right) \exp\left( \pi |tc - tb| \right) \|f(1 + it)\|_q\\
&\leq C_{a,q} \exp\left( \delta(1 - t^2 - \theta^2) \right) \exp\left( \pi |tc - tb| \right).
\end{align*}

Choosing \( \delta \) sufficiently large, we obtain
\[
\sup_{t \in \mathbb{R}} \max\{ |G(it)|, |G(1 + it)| \} < C_{p}.
\]

Consequently, by the maximum principle, \( |G(\theta)| < C_{p} \), i.e.,
\[
\left| \sum_r \tau\left( N^0 F_x(r) y_r \right) \right| < C_{p},
\]
which implies the desired result  by duality as in Proposition \eqref{dulaityofmaximalnorms}.

\end{proof}
\section{Noncommutative ergodic theorem for $SO^0(n,1)$ for $n>2$}\label{nevost}
\subsection{Embedding spherical averages in an analytical family:}
To establish the maximal inequality in \( \mathcal{D}_K(\mathcal M) \), we embed it into an analytic family of operators and then apply the method of analytic interpolation. Given \( x \in \mathcal{D}_K(\mathcal M) \),  and \( \alpha = a + ib \in \mathbb{C} \) with \( a > 0 \) and \( t > 0 \), consider the operator
\begin{equation}\label{Jalpha}
J_t^\alpha (x) = \frac{1}{\Gamma(\alpha)} \int_{0}^\infty \eta(s)s^{\alpha - 1} \pi(\sigma_{t(1+s)})x  \, ds.
\end{equation}
where \( \eta \in C^\infty_c(\mathbb{R}) \) is a cutoff function satisfying:
\[
\eta(s) =
\begin{cases}
1 & \text{for } s \in [0,1], \\
0 & \text{for } s \notin [-1/2, 3/2],
\end{cases}
\]
and \( \Gamma(\alpha) \) is the gamma function. Clearly $\alpha\mapsto J_t^\alpha(x)$ is analytic from the open right-half plane to $L_p(\mathcal M)$ for all $1\leq p<\infty.$

Now we extend the definition of \( J_t^\alpha(x) \) to the complex plane by analytic continuation, using the formula derived from integration by parts:
\begin{align*}
J_t^\alpha (x) 
&= \frac{\pi(\sigma_{2t})x}{\Gamma(\alpha+1)}-\frac{1}{\Gamma(\alpha+1)} \int_{0}^1s^\alpha\frac{d}{ds}\Big(\pi(\sigma_{t(1+s)}x)\Big)ds+\frac{1}{\Gamma(\alpha)}\int_1^\infty\eta(s)s^{\alpha-1}\pi(\sigma_{t(1+s)})x
\end{align*}

This representation is analytic in \( \alpha \) for \( \Re(\alpha)> -1 \), and defines an entire function in \( \alpha \) by further similar analytic continuation. 

Continuing similarly, \( J_t^\alpha f(x) \) can be extended to an entire function of \( \alpha \). Let us define 
\[M_t^\alpha(x)=t^{-\alpha}J_t^\alpha(x).\] We record the following two properties for future purpose.
\begin{itemize}\label{twoimport}
\item[1.]The distribution $f_a(s)=\frac{\eta(s)s^{a-1}}{\Gamma(a)}$ converges to the Dirac measure at $0$ as $a\to 0^+.$ Therefore, we have 
\[M_t^0(x)=\pi(\sigma_t)x.\]
\item[2.]Let $x\geq 0\in \mathcal{D}_K(\mathcal M)$ then \[M_t^1(x)=\frac{1}{t}\int_0^{\frac{3}{2}}\eta(s)\pi(\sigma_{t(1+s)})x\ ds\lesssim\frac{1}{t}\int_0^{3t}\pi(\sigma_s)x ds\lesssim \pi(\mu_{3t})x.\]
\end{itemize} 

\begin{prop}\label{prop2.5}Let $1<p\leq\infty.$ Then there exists $C_p(G)>0$ such that for every $x\in L_p(\mathcal M)$ we have 
\[\big\|\sup\limits_{t\geq 1}\!^+ M_t^{1+ib}x\big\|_p\leq C_p(G)\exp(\pi|b|)\|x\|_p.\]
\end{prop}

\begin{proof}
Without loss of generality, we may assume \( x \in\mathcal{S}(\mathcal{M})_{+} \). Let \( y\in L_{p'}(\mathcal{M})_{+} \), where \( p' \) is the conjugate exponent of \( p \), i.e., \( \frac{1}{p} + \frac{1}{p'} = 1 \). 

Using the standard estimates for the Gamma function (see, e.g., Page 79 of \cite{Stein70}), we have:
\[
\left| \frac{1}{\Gamma(a+ib)} \right| \leq C\exp(\pi|b|).
\] Now we observe that \[|\tau(M_t^{1+ib}(x)y)|=|\frac{1}{t^{1+ib}\Gamma(\alpha)} \int_{0}^\infty \eta(s)s^{ib} \tau(\pi(t(1+s))(x)y)  \, ds|\leq  C\exp(\pi|b|)\tau(\pi(\mu_{3t})xy).\]
For  finitely many $t_1,\dots,t_n>0$ we notice that for any positive sequence $(y_{t_i})\in L_{p^\prime}(\mathcal M)_{+}$ with $\|(y_{t_i})\|_{L_p(\mathcal M'\ell_1^n)}\leq 1$ we have
\begin{equation}|\sum\tau(M_t^{1+ib}(x)y_{t_i})|\leq \sum|\tau(M_t^{1+ib}(x)y_{t_i})|\leq C\exp(\pi|b|)\sum\tau(\pi(\mu_{3t_i})(x)y_{t_i})
\end{equation}
Therefore, by Theorem \eqref{thm2.1} we obtain that $|\sum\tau(M_t^{1+ib}(x)y_{t_i})\leq C\exp(\pi|b|)\|x\|_p.$ Moreover, any element in the unit ball of  $L_p(\mathcal{M},\ell_1^n)$ can be written as sum of eight positive elements in the same unit ball. Hence the theorem is proved by duality as in Propsoition \eqref{dulaityofmaximalnorms}.
\end{proof}
In the following proposition we prove an important estimate related to the spherical averages. It is useful to separate the estimate into two parts, corresponding to the \emph{principal series} and the \emph{complementary series} of positive-definite spherical functions.

Recall that the real spectrum of the $*$-algebra $M(G, K)$ is parametrized by:
\[
\Sigma = \{ z = i\lambda \mid \lambda \in \mathbb{R},\, \lambda \geq 0 \} \cup \{ z = s \mid s \in \mathbb{R},\, 0 \leq s \leq \rho_n \},
\]
where $\rho_n = \frac{n-1}{2}$ is a constant depending on the structure of $G$. As usual the character corresponding to the point $z \in \Sigma$ is denoted by $\varphi_z$.
Let $\chi_p$ be the characteristic function of the set
\[
\mathcal{E}_p = \{ z \in \mathcal{E} \mid z = i\lambda,\ \lambda \geq 1 \},
\]
and define $\chi_c(z) = 1 - \chi_p(z)$. Let $E_p$ and $E_c$ denote the spectral projection operators corresponding to the subsets $\mathcal{E}_p$ and $\mathcal{E}_c := \Sigma \setminus \mathcal{E}_p$ respectively.
\begin{lem}\cite{Hon16}\label{somel2bd}
Suppose \( F \) is a \( \mathcal M \)-valued function on \( \mathbb{R} \), which is smooth for \( t \) in an interval \( I \). Then for each \( \ell \) with \( \ell \leq |I| \), we have
\begin{equation} \label{eq:6.13}
|F(t)|^2 \leq 2\ell^{-1} \int_I |F(s)|^2\,ds + 2\ell \int_I |F'(s)|^2\,ds
\end{equation}
for each \( t \in I \).
\end{lem}
Let us define $M_t^\alpha\phi_z$ as
\begin{equation}\label{mtdefn}
M_t^\alpha\phi_z = \int_0^\infty\phi_z ({t(1 + s)})s^{\alpha - 1} \, ds.
\end{equation}
 The function is analytic for $\operatorname{Re} \alpha > 0$ and admits an analytic continuation to an entire function of $\alpha$.
\begin{lem}\cite{NevSt97}\label{mtdefnbd}  For \( |\lambda|\geq 1 \), \( t\geq 1 \), we have:
\begin{align*}
\text{(1)}\quad | M_t^\alpha \varphi_{i\lambda} | &\leq B(G) \, e^{-\frac{1}{2}\rho_{n}t }|\mathbf{c_n}(i\lambda)| \, |\lambda|^{-\operatorname{Re}\alpha}, \\
\text{(2)}\quad |\frac{d}{dt} M_t^\alpha(\varphi_{i\lambda}) | &B(G) \, e^{-\frac{1}{2}\rho_{n}t }|\mathbf{c_n}(i\lambda)| \, |\lambda|^{-\operatorname{Re}\alpha+1}.
\end{align*}
\end{lem}
\begin{prop}\label{mainestil2}For every \( x\in \mathcal{D}_K(\mathcal M) \),
 and for all \( a > -\frac{n}{2} + 1 \), we have:
 
\[
\big\|\sup\limits_{t\geq 1}\!^+ M_t^{a + ib}(x)\big\|_{2} \leq C_\alpha(G) \, \exp(\pi |b|) \, \| f \|_{L^2},
\]
provided that \( n(G) > 2 \), where \( n(G) \) denotes the dimension parameter associated to \( G \), and \( C_\alpha(G) \) is a constant depending on
 \( \alpha \) and \( G \).
\end{prop}
We divide the proof of the proposition into two parts: one regarding the principle series and another complementary series.
\begin{proof}[Proof of Proposition \eqref{mainestil2}:The principle series ]
Note that $\pi(m_K)$ is the unit of $A_\pi.$ Moreover, $\pi(m_K)$ is nothing but orthogonal projection onto the closed subspace $\{x\in L_2(\mathcal M):\pi(k)x=x\forall k\in K\}.$ Hence $E_p+E_c=\pi(m_K).$ Therefore, for $x\in\mathcal{D}_K(\mathcal {M}),$ we have
\[
\big\|\sup\limits_{t\geq 1}\!^+M_t^\alpha x\big\|_2 = \big\|\sup\limits_{t\geq 1}\!^+M_t^\alpha (E_px+E_cx)\big\|_2\leq\big\|\sup\limits_{t\geq 1}\!^+M_t^\alpha E_px\big\|+\big\|\sup\limits_{t\geq 1}\!^+M_t^\alpha E_cx\big\|_2.
\]

Applying the lemma to $g(t) = M_t E_p f(x)$, we obtain for all $t>0$ and $L>0$
\[
 |M_t E_p (x)|^2 \leq C L^{-1} \int_1^\infty |M_t E_p (x)|^2 dt + C L \int_L^\infty |\frac{d}{dt}M_t E_p (x)|^2 dt.
\]
Note that $\big\|\sup\limits_{t\geq 1}\!^+M_t^\alpha x\big\|_2\leq \big\|\sup\limits_{t\geq 1}\!^+|M_t^\alpha x|^2\big\|_1^{\frac{1}{2}}$
Hence we obtain that 
\begin{align}\label{l2esimate}
\bigl\|\sup_{t \ge 1}\!^+ M_t^{\alpha} E_p x\bigr\|_2^{2}
&\le C\,L^{-1}
      \Bigl\|\int_{1}^{\infty} |M_t E_p x|^{2}\,dt\Bigr\|_{1}
   +  C\,L
      \Bigl\|\int_{1}^{\infty}
               \Bigl|\tfrac{d}{dt} M_t E_p x\Bigr|^{2}\,dt
      \Bigr\|_{1}
      \\[4pt]\nonumber
&=  C\,L^{-1}
     \int_{1}^{\infty} \|M_t E_p x\|_{2}^{2}\,dt
   +  C\,L
     \int_{1}^{\infty}
          \Bigl\|\tfrac{d}{dt} M_t E_p x\Bigr\|_{2}^{2}\,dt.
\end{align}
Now we use the spectral theory of the algebra $M(G,K)$ to write the right-hand side using the spectral resolution of $E_p f$:
\[
\bigl\|\sup_{t \ge 1}\!^+ M_t^{\alpha} E_p x\bigr\|_2^{2}
\leq C L^{-1} \int_1^\infty \int_\Sigma |M_t^\alpha\phi_z|^2 \, d\nu_x(z) \, dt
+ C L \int_L^\infty \int_\Sigma |\frac{d}{dt}M_t^\alpha\phi_z|^2 \, d\nu_x(z) \, dt.
\]

Here $\nu_x$ is the spectral measure determined by $x$, and $\Sigma$ is the spectrum of the commutative $C^*$-algebra generated by $\pi(M(G,K))$. In particular, $E_p \nu_x$ is the measure determined by $E_p x$. The function in \eqref{mtdefn} defines a multiplier function associated (via the harmonic analysis on $M(G, K)$) with the operator $M_t^\alpha$. The function is analytic for $\operatorname{Re} \alpha > 0$ and admits an analytic continuation to an entire function of $\alpha$. By the functional calculus, the operator $M_t^\alpha$ corresponds on the Fourier transform side to the multiplier defined by the analytic continuation of the of operators $M_t^\alpha$.

Let us consider 
\[
I_k = [i2^k,\, i2^{k+1}) \subset\mathcal{E}_p,\quad \text{for } k \geq 0.
\]
Let \( E_k \) denote the spectral projection corresponding to the indicator function \( \chi_{I_k} \) of the interval \( I_k \). We aim to estimate the \( L_2 \)-norm of the operator \( E_k M_t^\alpha x \), where \( M_t^\alpha \) is the analytic family of operators defined earlier. Note that since \( E_k \) and \( M_t^\alpha \) commute (as both arise from the functional calculus of the commutative unital $C^*$-algebra \( \mathcal{A}_\pi \subset B(L_2(\mathcal{M}) \)), we have
\[
E_k M_t^\alpha x = M_t^\alpha E_k x.
\]
Hence, using the spectral theorem:
\[
\| E_k M_t^\alpha x \|_2^2 = \| M_t^\alpha E_k f \|_2^2 = \int_{I_k} |M_t^\alpha\varphi_z|^2\, d\nu_x(z),
\]
where \( \nu_x \) is the spectral measure associated to the vector \( x \).

Note that for \( |\lambda|\geq 1 \), we have the estimate:
\begin{equation}
|\mathbf{c_n}(i\lambda)| \leq C_n |\lambda|^{-\rho_n}.
\end{equation}
Therefore, by Lemma \eqref{mtdefn}
\[
\int_{I_k} | M_t^\alpha\varphi_z|^2 \, d\nu_x(z) 
\leq B(G)^2 \int_{I_k}\Big( e^{-\frac{1}{2}\rho_{n}t |}{|\mathbf{c_n}(z)||z|^{-\operatorname{Re}\alpha}} \Big)^2\, d\nu_x(z).
\]

Using the decay of \( \mathbf{c_n}(i\lambda) \), we estimate the above as
\[
\|M_t^\alpha E_kx\|_2^2=\int_{I_k} | M_t^\alpha\varphi_z|^2 \, d\nu_x(z) 
\leq B(G)^2 e^{-\rho_{n}t} {2^{-2k(\rho_n + \operatorname{Re}\alpha)} }
\|E_kx\|_2^2
\]

A similar computation yields that
\[
 \| E_k\frac{d}{dt}M_t^\alpha x \|_2^2  \leq B(G)^2 e^{-\rho_{n}t} {2^{-2k(\rho_n + \operatorname{Re}\alpha-1)} }
\|E_kx\|_2^2.
\]
Therefore, we obtain from \eqref{l2esimate} by choosing $L=2^{-k}$ 
\begin{equation}
\bigl\|\sup_{t \ge 1}\!^+ M_t^{\alpha} E_k x\bigr\|_2^{2}\leq 2CB(G)2^{-2k(\rho_n+\operatorname{Re}\alpha-\frac{1}{2})}\|E_kx\|_2^2.
\end{equation}
Note that $\operatorname{Re}\alpha=-\frac{n}{2}+1+\delta,$ $\delta>0$ and $\rho_n=\frac{1}{2}(n-1),$ hence for $k\geq 0$ we obtain $\bigl\|\sup_{t \ge 1}\!^+ M_t^{\alpha} E_k x\bigr\|_2^{2}\leq B2^{-2\delta k}\|E_kx\|_2^2.$ Summing it up over $k\geq 0$ we obtain that \[\bigl\|\sup_{t \ge 1}\!^+ M_t^{\alpha} E_p x\bigr\|_2\leq C_a(G)\|x\|_2\] where the constant $C_a(G)$ depends only on $a$ and $G.$ This completes the proof for the principle series.
\end{proof}

\subsection{Proof of Proposition \eqref{mainestil2}: The complementary series:
}

We follow the notation $N^\alpha(F_x)(t)\equiv N_t^\alpha(x)$ where $x\in \mathcal{D}_K(\mathcal M)$ by putting $\zeta(s)=1-\eta(2s)$ in \eqref{Nalpha} where $\eta$ is as in \eqref{Jalpha}.
\begin{lem}\label{lemma7.1}For all integer $m\geq 0$, there exist $C_m>0$ such that for all $x\in\mathcal{D}_K(\mathcal M)$ we have that 
\begin{equation}
\bigl\|\sup_{t \ge 1}\!^+N^{-m}_t(E_cx)\bigr\|_2\leq C_m\|E_cx\|_2.
\end{equation}
\end{lem}
\begin{proof}
The case $m=0$ is settled by Theorem \eqref{l2bdsphere} in Section \eqref{$L_2$-bound associated to spherical averages}. Hence it is enough to consider $m> 0.$ Note that by Proposition \eqref{prop4.1} it is enough to prove that $\|g_j(E_cx)\|_2\leq C_j\|E_cx\|_2$ for all $x\in \mathcal{D}_K(\mathcal{M}).$ Hence by Proposition \eqref{prophighorder} it is enough to show that the set for all $k\geq 1,$ there is a constant $C_k>0$ such that 
\begin{equation}\sup\{\int_1^\infty t^{2k-1}|\varphi_z^{(k)}(t)|^2:z\in \mathcal{E}_c\}\leq C_k.
\end{equation}
This bound have been proved in \cite[Proposition 8, part (4)]{Nevo97}.
\end{proof}
\begin{thm} For all $x\in \mathcal{D}_K(\mathcal M)$, and $a>-\frac{n}{2}+1$ we have that
\begin{equation}
\bigl\|\sup_{t \geq 1}\!^+ M^\alpha_t(E_c x)\bigr\|_2\leq C_a(G)\exp(\pi|b|)\|x\|_2
\end{equation}
provided $n(G)>2.$
\end{thm}
\begin{proof}
Let $-m<\operatorname{Re}\alpha<-m+1.$ We can follow the classical proof to obtain after applying integration by parts several times
\begin{equation}
\|\sup_{t \geq 1}\!^+M_t^\alpha(E_cx)\|_2\leq C_a\exp(\pi|b|)\Big(\sum_{k=0}^m\|\sup_{t \geq 1}\!^+N_{t}^{-k}E_cx\|_2+\|\sup_{t \geq 1}\!^+N_{t}^1(E_cx)\|_2\Big).\end{equation}
Therefore, the desired estimate follows from Lemma \eqref{lemma7.1} and Theorem \eqref{thm2.1}.
\end{proof}
\begin{thm}\label{thmconfu} Let $G=SO^0(n,1)$ and $n>2.$ Then for all $\frac{n}{n-1}<p\leq \infty$ we have for all $x\in L_p(\mathcal M)$
\begin{equation}
\|\sup_{t \geq 1}\!^+\pi(\sigma_t)x\|_p\leq C_p\|x\|_p
\end{equation}
\end{thm}
\begin{proof}
The proof of the theorem follows along the same line as proof of the Theorem \eqref{maximalsphericalforodd} by using Proposition \eqref{mainestil2} and Proposition \eqref{prop2.5} and complex interpolation. Hence we skip the proof.
\end{proof}

 \section{Ergodic theorem for Higher rank semisimple Lie group}\label{higherrank}
 Let $\mathcal M$ be a noncommutative probability space with a normal, faithful, tracial state $\tau.$ For, $1\leq p\leq\infty,$ denote $L_{p}^0(\mathcal M):=\{x\in L_p(\mathcal M):\tau(x)=0\}.$ Suppose we have a unitary representation $\pi:G\to L_2(\mathcal M).$ Then, clearly $ L_{2}^0(\mathcal M)$ is a closed subspace of $ L_2(\mathcal M)$ as $x\mapsto\tau(x)$ is a continuous linear functional. Moreover, if $\pi$ is induced by a dynamical system, then $L_{2}^0(\mathcal M)$ is invariant under $\pi$ as $\pi$ is trace-preserving. We call a $W^*$-dynamical system  $(\mathcal M,\tau, G,\pi)$ be a $W^*$-dynamical system to be ergodic if the set of $G$-invariant vectors in $L_2(\mathcal{M})$ is $\mathbb{C}.1.$ In what follows all actions will be considered ergodic.

\subsection{Structure theory for semisimple Lie groups}Let \( G \) be a connected semisimple Lie group with finite center and no nontrivial compact factors. Let \( \mathfrak{g} \) denote its Lie algebra, and suppose
\[
\mathfrak{g} = \bigoplus_{i=1}^N \mathfrak{g}_i,
\]
where each \( \mathfrak{g}_i \) is a simple ideal. Fix a Cartan involution \( \theta \) on \( \mathfrak{g} \), inducing the Cartan decomposition
\[
\mathfrak{g} = \mathfrak{k} \oplus \mathfrak{p},
\]
with \( \mathfrak{k} \) the \( (+1) \)-eigenspace and \( \mathfrak{p} \) the \( (-1) \)-eigenspace of \( \theta \). Let \( \mathfrak{a} \subset \mathfrak{p} \) be a maximal abelian subspace. The corresponding root system is denoted by \( \Sigma = \Sigma(\mathfrak{a}, \mathfrak{g}) \subset \mathfrak{a}^* \), and for each \( \alpha \in \Sigma \), denote the corresponding root space by \( \mathfrak{g}_\alpha \). Then the root space decomposition takes the form
\[
\mathfrak{g} = \mathfrak{m} \oplus \mathfrak{a} \oplus \bigoplus_{\alpha \in \Sigma} \mathfrak{g}_\alpha,
\]
where \( \mathfrak{m} \) is the centralizer of \( \mathfrak{a} \) in \( \mathfrak{k} \). Choose a set of simple roots \( \Delta \subset \Sigma \), which determines a system of positive roots \( \Sigma^+ \subset \Sigma \). Let \( \rho \in \mathfrak{a}^* \) denote the half-sum of the positive roots:
\[
\rho = \frac{1}{2} \sum_{\alpha \in \Sigma^+} \alpha.
\] Let \( K \) be the connected Lie subgroup of \( G \) with Lie algebra \( \mathfrak{k} \), so that \( K \) is a maximal compact subgroup of \( G \). Let \( W = W(\mathfrak{a}, \mathfrak{g}) \) denote the Weyl group associated with the root system. Let \( \mathfrak{a}^+ \) denote a chosen open positive Weyl chamber in \( \mathfrak{a} \), and let \( \overline{\mathfrak{a}^+} \) denote its closure. Define the subgroups \( A = \exp(\mathfrak{a}) \) and \( A^+ = \exp(\mathfrak{a}^+) \), with \( \overline{A^+} = \exp(\overline{\mathfrak{a}^+}) \). The Cartan (or polar) decomposition of \( G \) asserts that
\[
G = K \overline{A^+} K.
\]
For each regular element \( g \in G \), there exists a unique element \( H(g) \in \overline{\mathfrak{a}^+} \) such that
\[
g = k_1 \exp(H(g)) k_2, \quad \text{for some } k_1, k_2 \in K.
\] The Riemannian symmetric space \( G/K \) is equipped with a \( G \)-invariant metric \( d \) induced by the Killing form on \( \mathfrak{g} \). The restriction of which to \( \mathfrak{a} \) defines an inner product, and we have
\[
d(\exp(H) \cdot o, o) = \sqrt{\langle H, H \rangle}, \quad \text{for all } H \in \mathfrak{a},
\]
where \( o = [K] \) is the origin in \( G/K \). Let \( \{H_1, \dots, H_r\} \subset \mathfrak{a} \) be the basis dual to the chosen basis of \( \mathfrak{a}^* \) corresponding to the simple roots, with respect to the inner product.

The Haar measure \( m_G \) on \( G \) can be expressed in Cartan coordinates using the integration formula
\[
\int_G f(g)\, dm_G(g) = \int_K \int_{\overline{\mathfrak{a}^+}} \int_K f(k \exp(H) k')\, \xi(H)\, dm_K(k)\, dH\, dm_K(k'),
\]
for suitable functions \( f \).
Here \( \xi(H) \) is defined for \( H \in\overline{\mathfrak{a}^+} \) by
\[
\xi(H) = \prod_{\alpha \in \Sigma^+} \left( \sinh \alpha(H) \right)^{m_\alpha},
\]
where \( m_\alpha = \dim_{\mathbb{R}} \mathfrak{g}_\alpha \). The measure \( m_K \) refers to Haar measures on  \( K \), and \( dH \) denotes Lebesgue measure on \( \mathfrak{a} \). The function \( |\xi(H)| \) admits a \( W \)-invariant extension to all of \( \mathfrak{a} \), where \( W \) is the Weyl group associated with the root system \( \Sigma \) of \( \mathfrak{a} \) in \( \mathfrak{g} \). Note that \( \xi \) transforms under the action of \( W \) according to the rule:
\[
\xi(w(H)) = (-1)^{\det w} \cdot \xi(H), \quad \text{for all } w \in W.
\]
In particular, \( \xi \) is alternating under the Weyl group action.

\subsection{Ball averages:}Consider the family of radial (ball) averages on \( G \) defined via the Cartan decomposition. For each \( r > 0 \), define the set
\[
\widetilde{B}_r = \left\{ H \in \mathfrak{a} \,\middle|\, d\left( \exp(H) \cdot o, o \right) \leq r \right\} = \left\{ H \in \mathfrak{a} \,\middle|\, \|H\| \leq r \right\},
\]
where \( \|H\| \) denotes the norm induced by the Cartan–Killing form. We define a probability measure \( \beta_r \) on \( G \), supported on the double coset \( K \exp(\widetilde{B}_r) K \), by
\[
\beta_r = \frac{m_K * \left( \int_{\widetilde{B}_r} \delta_{\exp(H)} \, \xi(H) \, dH \right) * m_K}{\int_{\widetilde{B}_r} \xi(H) \, dH},
\]
where \( \delta_{\exp(H)} \) is the Dirac measure at \( \exp(H) \in A \subset G \), and \( \xi(H) \) is the density function from the Cartan integration formula.

Each \( \beta_r \) is a bi-\( K \)-invariant probability measure on \( G \), and under the canonical projection \( \tau: G \to G/K \), the measure \( \beta_r \) corresponds to the normalized Riemannian volume measure on the ball of radius \( r \) centered at \( o \) in the symmetric space \( G/K \).

\subsection{Spherical averages:}For each \( t > 0 \), let \( \sigma_t \) denote the unique bi-\( K \)-invariant probability measure on \( G \) whose projection to the symmetric space \( G/K \) under the canonical map \( \tau : G \to G/K \) coincides with the normalized surface measure on the geodesic sphere of radius \( t \) centered at the base point \( o = [K] \).

This measure can be written explicitly as
\[
\sigma_t = \frac{ \displaystyle\int_{\|H\| = t}| \xi(H)| \, (m_K * \delta_{\exp(H)} * m_K) \, d\omega_t(H) }{ \displaystyle\int_{\|H\| = t} |\xi(H)| \, d\omega_t(H) },
\]
where \( \omega_t \) denotes the rotation-invariant probability measure on the sphere of radius \( t \) in \( \mathfrak{a} \).
Alternatively, changing variables to the unit sphere in \( \mathfrak{a} \), this expression becomes
\[
\sigma_t = \frac{ \displaystyle\int_{\|H\| = 1} |\xi(tH)| \, \sigma_t^H \, d\omega_1(H) }{ \displaystyle\int_{\|H\| = 1} |\xi(tH)| \, d\omega_1(H) },
\]
where \( \sigma_t^H = m_K * \delta_{\exp(tH)} * m_K \), and \( \omega_1 \) is the rotation-invariant probability measure on the unit sphere \( \{ H \in \mathfrak{a} : \|H\| = 1 \} \).
\subsection*{Shell and Directional Averages}

\begin{defn}[Shell Averages]
For \( t \geq 1 \), define the shell average \( \gamma_t \) by
\[
\gamma_t = \int_0^1 \sigma_{t-s} \, ds,
\]
where \( \sigma_t \) denotes the bi-\(K\)-invariant spherical average of radius \( t \). For \( 0 \leq t \leq 1 \), we set
\[
\gamma_t = \gamma_1.
\]
\end{defn}

\begin{defn}[Directional Averages]
Let \( H \in \mathfrak{a}_+ \) be a unit vector (i.e., \( \|H\| = 1 \)). Define the directional spherical average of radius \( t \) as
\[
\sigma^H_t = m_K * \delta_{\exp(tH)} * m_K.
\]

Then for \( t \geq 1 \), define the directional shell average by
\[
\gamma^H_t = \int_0^1 \sigma^H_{t-s} \, ds,
\]
and for \( 0 \leq t \leq 1 \), we define
\[
\gamma^H_t = \gamma^H_1.
\]
\end{defn}

\begin{defn}[Directional Ball Averages]
Given \( b > 0 \) and \( t \geq 1 \), the directional ball average in direction \( H \in \mathfrak{a}_+ \) is defined as
\[
\beta^H_t = e^{-bt} \int_0^t e^{bs} \sigma^H_s \, ds.
\]
\end{defn}
\subsection{Property T} We recall the spectral methods applied to the commutative Banach \( * \)-algebra \( M(G, K) \), consisting of all bounded bi-\( K \)-invariant Borel measures on \( G \). 

Given a continuous unitary representation \( (\pi, \mathcal{H}_\pi) \) of \( G \), the action naturally extends to a \( * \)-homomorphism
\[
\pi: M(G, K) \to \mathrm{End}(\mathcal{H}_\pi).
\]

The spectral analysis of the image \( \pi(M(G, K)) \) involves studying the spectrum of its closure, which is a commutative Banach \( * \)-algebra. The Gelfand spectrum of this algebra can be viewed as a subset of the space of positive-definite spherical functions \( \varphi_\lambda \) on \( G \), which serve as the characters of \( M(G, K) \). Following Kazhdan, one defines the quantity
\[
\|\pi\|_T = \sup \{ \varphi_\lambda(\mu) \,:\, \varphi_\lambda \text{ is a nonconstant positive-definite spherical function} \},
\]
for \( \mu \in M(G,K) \), the space of bi-\( K \)-invariant probability measures on \( G \). For any unitary representation \( \tau \) of \( G \) without invariant vectors, the spectral theorem implies
\[
\|\pi(\mu)\| \leq \|\mu\|_T.
\] However, a fundamental distinction is whether \( \|\beta_1\|_T < 1 \) or \( \|\beta_1\|_T = 1 \). Groups for which \( \|\beta_1\|_T < 1 \) are said to have \textit{Kazhdan’s property T}. It is known that all connected semisimple Lie groups with finite center satisfy property T, unless they contain a factor locally isomorphic to \( \mathrm{SO}(n,1) \) or \( \mathrm{SU}(n,1) \). Moreover, Cowling gave a quantitative version: if \( G \) is a connected semisimple Lie group with no compact factors which has property T, then there exist constants \( b, B > 0 \) such that
\[
\|\beta_t\|_T \leq B e^{-bt}, \quad \text{for all } t \geq 0.
\]
\subsection{Spectral gap} Let \( G \) be a connected simple Lie group with finite center and no compact factors. Let \( (\pi, \mathcal{H}_\pi) \) be a strongly continuous unitary representation of \( G \). We say the representation $(\pi,\mathcal H_\pi)$ has a spectral gap if there exist constants \( \kappa_\pi > 0 \) and \( C > 0 \) such that for every \( H \in \mathfrak{a}_+ \), all vectors \( v, w \in \mathcal{H}_\pi \), and any \( \kappa < \kappa_\pi \), we have the decay estimate:
    \[
    \left| \left\langle \tau(m_K * \delta_{\exp H} * m_K) v, w \right\rangle \right| \leq C e^{-\kappa \rho(H)} \|v\| \|w\|.
    \]
\begin{defn}
Let \( \{\nu_t\}_{t \geq 0} \) be a family of probability measures on a Lie group \( G \). 

\begin{enumerate}
    \item \textbf{Monotonicity:} The family is called \emph{monotone} if there exists a constant \( B > 0 \) such that, for all \( t \geq 1 \),
    \[
    \nu_t \leq B \nu_{\lfloor t \rfloor + 1}
    \]
    in the sense of measures on \( G \), where \( \lfloor t \rfloor \) denotes the integer part of \( t \).

    \item \textbf{Rough Monotonicity:} The family is said to be \emph{roughly monotone} if there exist constants \( B > 0 \) and \( N \in \mathbb{N} \) such that, for all \( t \geq 1 \),
    \[
    \nu_t \leq B \sum_{k = \max(\lfloor t \rfloor - N, 0)}^{\lfloor t \rfloor + N} \nu_{\lfloor t \rfloor + k}.
    \]

    \item \textbf{Uniform Continuity:} The map \( t \mapsto \nu_t \) is called \emph{uniformly continuous} (with respect to the total variation norm) if for every \( \varepsilon > 0 \), there exists \( \delta > 0 \) such that
    \[
    \| \nu_{t + \tau} - \nu_t \| \leq \varepsilon
    \quad \text{whenever } 0 < \tau \leq \delta \text{ and } t \geq 1.
    \]

    \item \textbf{Uniform Hölder Continuity:} The map \( t \mapsto \nu_t \) is said to be \emph{uniformly Hölder continuous} with exponent \( 0 < \alpha \leq 1 \) if there exists a constant \( C > 0 \) such that
    \[
    \| \nu_{t + \varepsilon} - \nu_t \| \leq C |\varepsilon|^\alpha,
    \quad \text{for all } t \geq 1 \text{ and } 0 < \varepsilon \leq \tfrac{1}{2}.
    \]
\end{enumerate}

\noindent
If \( \nu_t \) is bi-\( K \)-invariant for all \( t \), then these regularity conditions are determined solely by the natural projections $\widetilde{\nu_t}$ of $\nu_t$ to measures on $\overline{A^+}.$
\end{defn}

\begin{thm}\label{thm4ofnevostmar}
Let $(\mathcal M,\tau, G,\pi)$ be a $W^*$-dynamical system where $\mathcal M$ is a noncommutative probability space with a faithful normal tracial state $\tau$ and $G$ is a connected semisimple Lie group with finite center and no compact factors. Assume that the averages $\nu_t$ are bi-$K$-invariant, roughly monotone, uniformly Holder continuous with exponent $a$, and satisfy $\|\pi_0(\nu_t)\|\leq B\exp(-t\theta),$ where $\theta>0.$ The the following statements hold.
\begin{itemize}
	\item[(i)] $\big\|\sup\limits_{t\geq 1}\!^+\pi(\nu_t)x\big\|_p\leq C_p\|x\|_p$ for $1<p<\infty.$
	\item[(ii)] 
	Let $1<p<\infty,$ then $\pi(\nu_t)x\to \tau(x)1$, a.u. as $t\to\infty$ for all $x\in L_p(\mathcal M).$
	\item[(iii)] For all $x\in L_p(\mathcal M),$ $1<p<\infty,$
	\[\big\|\sup\limits_{t\geq 1}\!^+\exp\Big(\frac{ua\theta t}{4}\Big)\Big(\pi(\nu_t)x-\tau(x).1\Big)\big\|_r\leq B\|x\|p\] provided $p,r$ and $0<u<1$ satisfy $\frac{1}{p}=\frac{1-u}{q}$ and $\frac{1}{r}=\frac{1-u}{q}+\frac{u}{2}$ for some $1<q<\infty.$
	\item[(iv)]The averages $\pi(\nu_t)x$ converges $\tau(x).1$ in an exponential rate, i.e. for all $t>0$ and $x\in L_p(\mathcal{M})$ we have \begin{equation}\label{meanexponen}
	\|\pi(\nu_t)x-\tau(x).1\|_r\leq C\|x\|_p\exp\Big(-\frac{ua\theta t}{4}\Big)\end{equation} where $C>0$ is a constant not depending on $x.$ Moreover, we have fro all $t_0>1$
	\begin{equation}\label{meanexponen1}
	\big\|\sup\limits_{t\geq t_0}\!^+\Big(\pi(\nu_t)x-\tau(x).1\Big)\|_r\leq B\|x\|_p\exp\Big(-\frac{ua\theta t_0}{4}\Big)
	\end{equation} for all $t>t_0.$
\end{itemize}
\end{thm}
\begin{proof}

(i). Since $\|\pi_0(\nu_t)\|_{2\to 2}\leq B\exp(-t\theta).$ Therefore, we have 
\[\sum_{n=1}^\infty\|\pi(\nu_n)x\|_2^2\leq c^2\|x\|_2^2\] for all $x\in L_2^0(\mathcal M).$
 Note that for all $1\leq j\leq n,$ $|\pi(\nu_j)x|^2\leq \sum_{j=1}^n|\pi(\nu_j)x|^2$. Therefore, by operator monotonicity of $t\mapsto t^{1/2}$, we have for all $1\leq j\leq n,$ 
\[|\pi(\nu_j)x|\leq\Big(\sum_{j=1}^n|\pi(\nu_j)x|^2\Big)^{1/2}. \] Hence there exists contractions $u_j\in \mathcal M$ such that $\pi(\nu_j)x=u_j\Big(\sum_{j=1}^n|\pi(\nu_j)x|^2\Big)^{1/2}.$
Thus clearly, we have 
\[\big\|\sup\limits_{1\leq j\leq n}\!^+\pi(\nu_j)x\big\|_2\leq \Big(\tau\Big(\sum_{j=1}^n(\pi(\nu_j)x)^2\Big)\Big)^{1/2}=\Big(\sum_{j=1}^n\|\pi(\nu_j)x\|_2^2\Big)^{1/2}\leq c\|x\|_2. \]
 Therefore we have for some constant $K>0$ 
\begin{equation}\label{S}\big\|\sup\limits_{n\geq 1}\!^+\pi(\nu_n)x\big\|_2\leq K\|x\|_2\end{equation} for all $x\in L_2^0(\mathcal M).$

For all $x\in L_2(\mathcal M)$ note that $x-\tau(x)1\in L_2^0(\mathcal M).$ Therefore for $x\in L_2(\mathcal M)$ we have by triangle inequality on $L_2(\mathcal M;\ell_\infty)$
\begin{equation}\label{S1}\big\|\sup\limits_{n\geq 1}\!^+\pi(\nu_n)x\big\|_2\leq \big\|\sup\limits_{n\geq 1}\!^+(\pi(\nu_n)x-\tau(x)1)\big\|_2+|\tau(x)|\big\|\sup\limits_{n\geq 1}\!^+\pi(\nu_n)1\big\|_2 \end{equation}.

 Note that $\big\|\sup\limits_{n\geq 1}\!^+\pi(\nu_n)1\big\|_2=1.$ Therefore, from \eqref{S} and \eqref{S1} we obtain
 
\begin{equation}\big\|\sup\limits_{n\geq 1}\!^+\pi(\nu_n)x\big\|_2\leq K\|x-\tau(x)1\|_2+|\tau(x)|\leq K\|x\|_2+K|\tau(x)|+|\tau(x)|\leq (2K+1)\|x\|_2.
\end{equation}

Now take $x\in L_2(\mathcal M)_{+}.$ We have $\nu_t\leq B\sum_{k=\text{max}([t]-N,0)}^{[t]+N}\nu_{[t]+k}$ as the family $\nu_t$ s roughly monotone. Therefore,
\[\pi(\nu_t)x\leq B\sum_{k=\text{max}([t]-n,0)}^{[t]+N}\pi(\nu_{[t]+k})x.\] Let $a\in L_2(\mathcal M)_{+}$ be such that $\pi(\nu_n)x\leq a$ for all $n\geq 1$.
Thus we have from the above
\[\pi(\nu_t)x\leq B(2N+1)a \] for all $t\geq 1.$ Taking infimum over all $a$'s chosen above, we obtain
\[\big\|\sup\limits_{t\geq 1}\!^+\pi(\nu_t)x\big\|_2\leq B(2N+1)\big\|\sup\limits_{n\geq 1}\!^+\pi(\nu_n)x\big\|_2\leq c_2\|x\|_2.\]

For $p>1,$ define the operator $\pi(\nu_t^0):L_p(\mathcal M)\to L_p(\mathcal M)$ as $\pi(\nu_t^0)x=\pi(\nu_t)x-\tau(x)1.$ Then we have that $\|\pi(\nu_t^0)\|_{1\to 1}\leq 2,$ $\|\pi(\nu_t^0)\|_{\infty\to\infty}\leq 2$ and $\|\pi(\nu_t^0)\|_{2\to 2}\leq C\exp(-\theta t)\|x\|_2.$ Therefore, by Riesz-Thorin complex interpolation, we have
\[\|\pi(\nu_t^0)\|_{p\to p}\leq B_p\|\pi(\nu_t^0)\|_{1\to 1}^{(2/p-1)}\|\pi(\nu_t^0)\|_{2\to 2}^{2(1-1/p)},\ \ 1\leq p\leq 2\] and 
\[\|\pi(\nu_t^0)\|_{p\to p}\leq B_p\|\pi(\nu_t^0)\|_{1\to 1}^{2/p}\|\pi(\nu_t^0)\|_{\infty\to\infty}^{(1-2/p)},\ \ 2\leq p\leq \infty.\]
Therefore, we obtain that for all $1<p\leq\infty$ $\|\pi(\nu_t^0)\|_{p\to p}\leq B_p\exp(-\theta_pt)$ where $\theta_p=2\theta(1-1/p)$ for $1<p\leq 2$ and $\theta_p=2\theta/p$ if $2\leq p\leq\infty.$ 
The maximal inequality now follows by repeating the proof for $p=2.$

(ii). First we show b.a.u. convergence. Fix $1<p<\infty.$ For $1<p<\infty,$ we have from (i), for all $x\in L_p(\mathcal M)$
\[\|\pi(\nu_t)x-x\|_p\leq B_p\exp(-t\theta_p)\|x\|_p\] for all $t>0.$ Now the family $\nu_t$ is uniformly H\"older continuous with exponent $0<a\leq 1.$ Therefore the map $t\mapsto \pi(\nu_t)x$ is continuous from $[1,\infty)$ to $L_p(\mathcal M).$ For any $t_0\geq 1$ we have
\begin{align*}
	\int_{t\geq t_0}\|\pi(\nu_t)x-\tau(x)\|_{p}^p &\leq B_p^p\|x\|_p^p\int_{t_0}^\infty\exp(-tp\theta_p)dt\\
	& =\frac{1}{p\theta_p}B_p^p\|x\|_p^p\exp(-p\theta_pt_0)\to 0,\ \text{as}\ t\to \infty.
\end{align*}
By \Cref{ptwiselem}, $\pi(\nu_t)x\to \tau(x)$, b.a.u. as $t\to\infty.$

Now, we return to a.u. convergence. Note that for all $x\in L_\infty(\mathcal M)\cap L_2(\mathcal M)$, we have the inequality
\[\|\pi(\nu_t)x-\tau(x)\|_2\leq B\exp(-t\theta)\|x-\tau(x)\|_2.\] Choose an increasing sequence $t_i\in\mathbb R^{+}$ such that $\sum_{i=0}^\infty\exp(-\theta t_i/2)<\infty.$ For any $y\in L_2^0(\mathcal M)$ define 
\[C(y):=\sum_{i=1}^\infty\Big|\exp(\theta t_i/2)\pi(\nu_{t_i})y\Big|^2.\] Then $C(y)^{\frac{1}{2}}$ is a well-defined and an element of $L_2(\mathcal M),$ because
\[\tau(C(y))=\sum_{i=1}^\infty\exp(\theta t_i)\tau(|\pi(\nu_{t_i})y|^2)\leq B^2\|y\|_2^2\sum_{i=1}^\infty\exp(-\theta t_i)\] which is finite. 

Note that $|\exp(\theta t_i/2)\pi(\nu_{t_i})(x-\tau(x))|^2\leq C(x-\tau(x)).$ Therefore, we have 
\[\|\pi(\nu_{t_i})(x-\tau(x))\|_2\leq \tau(C(x-\tau(x)))^{\frac{1}{2}}\exp(-\theta t_i/2).\]
For $x\in L_\infty(\mathcal M)$ we show that $(\pi(\nu_t)x)_{t>0}\in L_2(\mathcal M;\ell_\infty^c).$ Note that 
\begin{align}\label{nsm1}
\|(\pi(\nu_t)x-\tau(x))_{t>0}\|_{L_2(\mathcal M;\ell_\infty^c)}&\leq \|(\pi(\nu_t)x-\pi(\nu_{t_i})x)_{t>0}\|_{L_2(\mathcal M;\ell_\infty^c)}+\|\pi(\nu_{t_i})x-\tau(x)\|_2\\\nonumber
&\leq \|(\pi(\nu_t)x-\pi(\nu_{t_i})x)_{t>0}\|_{L_2(\mathcal M;\ell_\infty^c)}+\tau(C(x-\tau(x)))^{\frac{1}{2}}\exp(-\theta t_i/2).
\end{align}
Now we choose the sequence $(t_i)_{i\geq 1}$ as the following. For each $n\in\mathbb N$ divide the interval $[n,n+1]$ into $2+[\exp(n\theta/4)]$ intervals of equal length, and take $(t_i)_{i\geq 1}$ to be the increasing sequence consisting of their endpoints. With this choice of $(t_i)_{\geq 1}$ we have
\[\sum_{n=0}^\infty\sum_{n\leq t_i<n+1}\exp(-\theta t_i/2)\leq\Big([\exp(n\theta/4)]+2\Big)\exp(-\theta n/2).\]
For $[t]=n,$ $t$ is in some subinterval of length $1/2$ (with some modulus $0<a\leq 1$). we have the estimate
\begin{align*}
	\|\pi(\nu_t)x-\pi(\nu_{t_i})x\|_\infty\leq \|x\|_\infty\Big([\exp(n\theta/4)]+2\Big)^{-a}\leq \|x\|_\infty\exp(-[t]a\theta/4).
\end{align*}
Therefore, from \eqref{nsm1} we obtain
\begin{equation}
\|(\pi(\nu_t)x)_{t>0}\|_{L_2(\mathcal M;\ell_\infty^c)}\leq \Big(\|x\|_\infty+\tau(C(x-\tau(x)))^{\frac{1}{2}}\Big)\sup\limits_{t\geq 1}\exp(-[t]a\theta/4)<\infty.
\end{equation}
Therefore, from \Cref{ct1} and the fact that $\pi(\nu_t)x\to \tau(x)$ b.a.u. as $t\to\infty$ for all $x\in L_\infty(\mathcal M),$ we have $\pi(\nu_t)x\to \tau(x)$ a.u. as $t\to\infty.$ Since $L_\infty(\mathcal M)$ is dense in $L_2(\mathcal M)$ we must have by \Cref{ptwiselem} that $\pi(\nu_t)x\to\tau(x)$ a.u. as $t\to\infty$ for all $x\in L_2(\mathcal M).$

(iii).
Note that by repeating the argument in (ii) it follows that 
\[\big\|\sup\limits_{t\geq 1}\!^+ \exp(a\theta t/4)\pi(\nu_t)x\big\|_2\leq\Big(\|x\|_\infty+\tau(C(x-\tau(x)))^{\frac{1}{2}}\Big).\]
Note that it was observe in (ii) that for some constant $B^\prime>0$ which is independent of $x\in L_\infty(\mathcal M)$ we have
\[\tau(C(x-\tau(x)))^{\frac{1}{2}}\leq B^\prime\|x-\tau(x)\|_2\leq B\|x\|_\infty\] where $B>0$ is also independent of $x\in L_\infty(\mathcal M).$ Hence for all $x\in L_\infty(\mathcal M)$ there exists a constant $A>0$ such that
\begin{equation}\label{fundal2ineq}
\big\|\sup\limits_{t\geq 1}\!^+ \exp(a\theta t/4)(\pi(\nu_t)x-\tau(x))\big\|_2\leq A\|x\|_\infty
\end{equation} Also we have a maximal inequality for $1<q\leq \infty$
\begin{equation}\label{fundalqineq}
\big\|\sup\limits_{t\geq 1}\!^+(\pi(\nu_t)x-\tau(x))\|_q\leq C_q\|x\|_q.
\end{equation}

Suppose \(1 < q \leq 2\), and fix \(u \in (0,1)\). Define exponents \(p\) and \(r\) by
\[
\frac{1}{p} = \frac{u}{q}, \qquad \frac{1}{r} = \frac{u}{q} + \frac{1 - u}{2}.
\]
Let \(x \in L_p(\mathcal{M})\) satisfy \(\|x\|_p < 1\), and let \(y \in L_{r'}(\ell_1)\) with \(\|y\|_{L_{r'}(\ell_1)} < 1\).

Let \(g = (g_j)_j\) be a bounded continuous function on the strip
\[
S := \{ z \in \mathbb{C} : 0 \leq \Re z \leq 1 \},
\]
analytic in the interior of \(S\), such that \(g(u) = y\) and
\[
\sup_{s \in \mathbb{R}} \max \left\{ \|g(is)\|_{L_2(\mathcal{M}, \ell_1)},\ \|g(1+is)\|_{L_{q'}(\ell_1)} \right\} < 1.
\]
Define $f(z):=u|x|^{\frac{rz}{q}+\frac{(1-z)r}{2}}.$
Define a scalar-valued function on the strip by
\[
G(z) := \sum_j \exp\left( \frac{a \theta (1-z) t_j}{4} \right) \tau\left( \left( \pi(\nu_{t_j}) f(z) - \tau(f(z)) \right) g_j(z) \right).
\]
Then \(G\) is analytic in the open strip \(0 < \Re z < 1\), and continuous on the closure \(\overline{S}\).

We estimate \(G\) on the boundary of the strip. For real \(s \in \mathbb{R}\), we have
\[
|G(is)| \leq \left\| \left( \exp\left( \frac{a \theta(1 -i s) t_j}{4} \right) (\pi(\nu_{t_j}) f(is) - \tau(f(is))) \right)_j \right\|_{L_2(\mathcal{M}, \ell_\infty)}
\cdot \|g(is)\|_{L_2(\mathcal{M}, \ell_1)}.
\]
Since \(|\exp(i\cdot)| = 1\), we obtain by Lemma \eqref{gettingtomodulus}
\[
|G(is)| \leq \left\| \left( \exp\left( \frac{a \theta t_j}{4} \right) (\pi(\nu_{t_j}) f(is) - \tau(f(is))) \right)_j \right\|_{L_2(\mathcal{M}, \ell_\infty)} \cdot \|g(is)\|_{L_2(\mathcal{M}, \ell_1)}. 
\] 
Using the boundedness assumption on \(g\), and \eqref{fundal2ineq} we obtain
\[
|G(is)| \leq C \|f(is)\|_2 \leq C \|x\|_2.
\]

A similar computation, using the assumption on \(g\) and inequality~\eqref{fundalqineq}, yields the following estimate for the upper boundary of the strip. For all \(s \in \mathbb{R}\), we have
\begin{align*}
|G(1 + is)|
&\leq \left\| \left( \exp\left( \frac{a \theta (-is) t_j}{4} \right) \left( \pi(\nu_{t_j}) f(1 + is) - \tau(f(1 + is)) \right) \right)_j \right\|_{L_q(\mathcal{M}, \ell_\infty)} \cdot \|g(1 + is)\|_{L_{q'}(\mathcal{M}, \ell_1)} \\
&= \left\| \left( \pi(\nu_{t_j}) f(1 + is) - \tau(f(1 + is)) \right)_j \right\|_{L_q(\mathcal{M}, \ell_\infty)} \cdot \|g(1 + is)\|_{L_{q'}(\mathcal{M}, \ell_1)} \\
&\leq C_q.
\end{align*}
Hence by Hadamard's there lines lemma we obtain that \[|G(u)|= \Big|\sum_j \exp\left( \frac{a \theta (1-u) t_j}{4} \right) \tau\left( \left( \pi(\nu_{t_j}) x - \tau(f(z)) \right) y \right)\Big|\leq C_q.\] The required result follows from the duality.

Note that the inequality \eqref{meanexponen} is straightforward to verify. Therefore, we focus on proving \eqref{meanexponen1}.

Consider a factorization of the form
\[
\exp\left(\frac{ua\theta t}{4}\right)(\pi(\nu_t)x - \tau(x) \cdot 1) = a y_t b,
\]
where \( a, b \in L_{2r}(\mathcal{M}) \) and \( y_t \in \mathcal{M} \) is a bounded family.

Observe that for \( t > t_0 \), we have
\[
\pi(\nu_t)x - \tau(x) \cdot 1 = \exp\left(-\frac{ua\theta t_0}{4}\right) a \exp\left(-\frac{ua\theta (t - t_0)}{4}\right) y_t b.
\]
Taking the \( L_r \)-norm, we obtain
\[
\|\pi(\nu_t)x - \tau(x) \cdot 1\|_r \leq \exp\left(-\frac{ua\theta t_0}{4}\right) \|a\|_{2r} \sup_t \|y_t\| \cdot \|b\|_{2r}.
\]
Taking the infimum over all such factorizations yields the desired estimate by applying part (iii).
This completes the proof of the theorem.
\end{proof}
In view of the above theorem and the fact that $\beta_t$ satisfies requried conditions, we obtain the following corolloary.
\begin{cor} Suppose \( G \) is a connected semisimple Lie group with finite center and no compact factors, and that \( G \) satisfies Kazhdan's property (T). Let $(\mathcal M,\tau, G,\pi)$ be a $W^*$-dynamical system where $\mathcal M$ is a noncommutative probability space with a faithful normal tracial state $\tau$.   Then there exists a positive constant \( b(G) \) such that for every \( 0<\theta < b(G) \), the following holds:
\begin{itemize}
	\item[(i)] $\big\|\sup\limits_{t\geq 1}\!^+\pi(\beta_t)x\big\|_p\leq C_p\|x\|_p$ for $1<p<\infty.$
	\item[(ii)] 
	Let $1<p<\infty,$ then $\pi(\beta_t)x\to \tau(x)1$, a.u. as $t\to\infty$ for all $x\in L_p(\mathcal M).$
	\item[(iii)] For all $x\in L_p(\mathcal M),$ $1<p<\infty,$
	\[\big\|\sup\limits_{t\geq 1}\!^+\exp\Big(\frac{ua\theta t}{4}\Big)\Big(\pi(\beta_t)x-\tau(x).1\Big)\big\|_r\leq B\|x\|p\] provided $p,r$ and $0<u<1$ satisfy $\frac{1}{p}=\frac{1-u}{q}$ and $\frac{1}{r}=\frac{1-u}{q}+\frac{u}{2}$ for some $1<q<\infty.$
	\item[(iv)]The averages $\pi(\beta_t)x$ converges $\tau(x).1$ in an exponential rate, i.e. for all $t>0$ and $x\in L_p(\mathcal{M})$ we have \begin{equation}\label{meanexponen}
	\|\pi(\beta_t)x-\tau(x).1\|_r\leq C\|x\|_p\exp\Big(-\frac{ua\theta t}{4}\Big)\end{equation} where $C>0$ is a constant not depending on $x.$ Moreover, we have fro all $t_0>1$
	\begin{equation}\label{meanexponen1}
	\big\|\sup\limits_{t\geq t_0}\!^+\Big(\pi(\beta_t)x-\tau(x).1\Big)\|_r\leq B\|x\|_p\exp\Big(-\frac{ua\theta t_0}{4}\Big)
	\end{equation} for all $t>t_0.$
\end{itemize}
\end{cor}
\begin{thm}\label{integerradii}
Let $G$ be a connected semisimple Lie group with finite center and no compact factors, acting on $\mathcal M$ in a ergodic trace-preserving way. Let $\nu_t$ be any family of bi-$K$-invariant averages on $G$, satisfying the spectral decay estimate
\[
\|\pi_0(\nu_t)\| \leq B e^{-\theta t}, \quad \text{for some } \theta > 0.
\]

Then the following hold.

\begin{enumerate}
    \item (\textbf{Exponential Maximal Inequality}) \\
    For every $x \in L^p(X)$ with $1 < p < \infty$, we have
    \[
    \left\| \sup_{n \geq 1}\!^+  e^{\frac{\theta_p n}{2}} \left( \pi(\nu_n)x - \tau(x).1  \right)  \right\|_{p} \leq C_p \|x\|_{p},
    \]
    where
    \[
    \theta_p =
    \begin{cases}
    2\theta \left(1 - \frac{1}{p} \right), & \text{if } 1 \leq p \leq 2, \\
    \frac{2\theta}{p}, & \text{if } p \geq 2.
    \end{cases}
    \]

    \item (\textbf{Exponential Convergence Rate}) \\
    The averages $\pi(\nu_n)x$ converges $\tau(x).1$ in an exponential rate, i.e. for all $n\in\mathbb{N}$ and $x\in L_p(\mathcal{M})$ we have \begin{equation}\label{meanexponen}
	\|\pi(\nu_n)x-\tau(x).1\|_p\leq C\|x\|_p\exp\Big(-{\frac{1}{2}\theta_p^\prime n}\Big)\end{equation} where $C>0$ is a constant not depending on $x.$ Moreover, we have fro all $n_0>1$
	\begin{equation}\label{meanexponen1}
	\big\|\sup\limits_{n\geq n_0}\!^+\Big(\pi(\nu_t)x-\tau(x).1\Big)\|_r\leq B\|x\|_p\exp\Big(-{\frac{1}{2}\theta_p^\prime n_0}\Big)\Big)
	\end{equation} for all $n>n_0.$

    \item (\textbf{Generalization to Sparse Sequences}) \\
    A similar result holds for any sequence $\nu_{t_n}$ provided that
    \[
    \sum_{n=1}^{\infty} e^{-\frac{\theta_p t_n}{2}} < \infty.
    \]
   
\end{enumerate}

\end{thm}
\begin{proof}
As in the proof of Theorem \eqref{thm4ofnevostmar} we obtain for any $x\in L_(\mathcal M)_{0}$ we have $\|\pi(\nu_t)x\|_p\leq B_p\exp(-{\theta_pt})\|x\|_p$ for some $\theta_p>0.$ Let $1<p<\infty.$ Define \[C_p^{\epsilon}(x):=\sum_{n=0}^\infty \Big(\exp((1-\epsilon)\theta_pn)|\pi(\nu_nx)|^p\Big).\] Note that $C_p^\epsilon(x)$ exists as $\tau(C_p^\epsilon(x))=\sum_{n=0}^\infty\exp((1-\epsilon)\theta_pn)\|\pi(\nu_n)x\|_p^p\leq (B_p(\epsilon)\|x\|_)^p<\infty.$ Now proceeding again as in the proof of Theorem \eqref{thm4ofnevostmar} we obtain 
\begin{equation}
 \left\|\sup_{n \geq 1}\!^+\exp((1-\epsilon)\theta_pn)\pi(\nu_n)x\right\|_p\leq \tau(C_p^{\epsilon}(x))^{\frac{1}{p}}\leq B_p\|x\|_p
\end{equation}
This proves part (1) of the theorem. The rest of the proof of the theorem is very similar to Theorem \eqref{thm4ofnevostmar}. We skip the proof.
\end{proof}	
 
 \begin{thm}\cite{MaNS00}
Let \( G \) be a connected semisimple Lie group with finite center and no compact factors. The following families of bi-\(K\)-invariant probability measures are both roughly monotone and uniformly Hölder continuous:
\begin{enumerate}
    \item The spherical shell averages \( \gamma_t \),
    \item The directional ball averages \( \beta^H_t \) and directional shell averages \( \gamma^H_t \),
    \item The convolutions \( m_K * \nu^L_t * m_K \), where \( \nu^L_t \) denotes either the ball or shell averages on a connected semisimple subgroup \( L \subset G \) (with no compact factors), and is invariant under the Cartan involution.
\end{enumerate}
Moreover, in any measure-preserving action of \( G \) for which the operator norm satisfies the decay estimate
\[
\| \pi_0(\nu_t) \| \leq B e^{-\theta t}, \quad \text{for some } \theta > 0,
\]
each of the above averaging families satisfies the maximal inequality and pointwise convergence properties as established in Theorem \eqref{thm4ofnevostmar}.
\end{thm}
\begin{rem}
It has been proved in \cite{MaNS00} that if the action has a spectral gap then the exponential decay condition is satisfied for $\gamma_t,$ $\beta_t^H,$ $m_K*\nu_t^L*m_K$ as in the above theorem and spherical measures $\sigma_t.$ We refer \cite{Ne06} for various examples.
\end{rem}

\textbf{Acknowledgement:} The first author is partially supported by National Natural Science Foundation of
China (No. 12071355, No. 12325105, No. 12031004, No. W2441002). The second named author thanks the DST-INSPIRE Faculty Fellowship
DST/INSPIRE/04/2020/001132 and Prime Minister Early Career Research Grant
Scheme ANRF/ECRG/2024/000699/PMS.  The second named author is very thankful to Swagato K.Ray and Mithun Bhowmik for various discussion around topics related to harmonic analysis on symmetric spaces.

\end{document}